\documentclass[12pt, reqno, letterpaper]{amsart}
\usepackage[margin=1in]{geometry}

\usepackage{lmodern}
\usepackage[T1]{fontenc}

\usepackage{amssymb, amsmath, amsfonts, amsthm, mathrsfs, tikz-cd, hyperref}
\usepackage{cite}
\usepackage{dutchcal, bm}

\usepackage{fancyhdr} 
\usepackage{color}
\definecolor{linkcol}{rgb}{0,0,0.6}
\definecolor{hrefcol}{rgb}{0,0.5,0.5}
\definecolor{citecol}{rgb}{0.6,0,0.6}
\definecolor{amber}{rgb}{0.5,0.4,0.0}
\hypersetup{colorlinks=true,
            linkcolor=linkcol,
            filecolor=magenta,      
            urlcolor=hrefcol,
            citecolor=citecol}
\tikzcdset{arrow style=tikz, diagrams={>=stealth}}

\numberwithin{equation}{section}
\newcommand\numeq{\stepcounter{equation}\tag*{{\sc(\arabic{section}.\arabic{equation})}}}

\makeatletter 
\renewcommand\th@plain{\slshape}
\makeatother

\newtheoremstyle{result}
  {1.7ex}
  {1.7ex}
  {\slshape}
  {0pt}
  {\scshape}
  {.}
  { }
  {\thmname{#1}\thmnumber{ #2}\thmnote{\if&#3& #3\else { }$-$ #3\fi}}

\newtheoremstyle{unique}
  {1.7ex}
  {1.7ex}
  {}
  {0pt}
  {\scshape}
  {.}
  { }
  {#3}

\newtheoremstyle{other}
  {1.7ex}
  {1.7ex}
  {}
  {0pt}
  {\scshape}
  {.}
  { }
  {\thmname{#1}\thmnumber{ #2}\thmnote{(#3) }}

\theoremstyle{result}
\newtheorem{THM}{Theorem}[section]
\newtheorem{PROP}[THM]{Proposition}
\newtheorem{LEM}[THM]{Lemma}
\newtheorem{COR}[THM]{Corollary}

\theoremstyle{unique}
\newtheorem{UNI}{}

\theoremstyle{other}
\newtheorem{DEF}[THM]{Definition}
\newtheorem{DEFS}[THM]{Definitions}
\newtheorem{NOT}[THM]{Notation}
\newtheorem{CONV}[THM]{Convention}

\newtheorem{RMK}[THM]{Remark}
\newtheorem{RMK!}[THM]{Important Remark}

\newtheorem{EXP}[THM]{Example}

\def\bb{\mathbb}

\def\0{\varnothing}
\def\scr{\mathcal}
\def\<{\langle}
\def\>{\rangle}

\newcommand{\keywd}[1]{{\fontseries{b}\selectfont{\boldmath#1}}}

\title{Revisiting the Cohen-Jones-Segal construction in Morse-Bott theory}
\author{Ciprian M. Bonciocat}
\address{Department of Mathematics, Stanford University, 450 Jane Stanford Way, Building 380,
Stanford, CA 94305-2125, USA.}
\email{ciprianb@stanford.edu}

\begin{document}

\begin{abstract}
    In 1995, Cohen, Jones and Segal proposed a method of upgrading any given Floer homology to a stable homotopy-valued invariant. For a generic pseudo-gradient Morse-Bott flow on a closed smooth manifold $M$, we rigorously construct the alleged stable normal framings, which are an essential ingredient in their construction, and give a rigorous proof that the resulting stable homotopy type recovers $\Sigma^\infty_+ M$. We further show that other systems of compatible stable normal framings recover Thom spectra $M^E$, for all reduced $KO$-theory classes $E$ on $M$. Our paper also includes a construction of the smooth corner structure on compactified moduli spaces of broken flow lines with free endpoint, a formal construction of Piunikhin-Salamon-Schwarz type continuation maps, and a way to relax the stable normal framing condition to orientability in orthogonal spectra.
\end{abstract}
\maketitle

\setcounter{tocdepth}{1}
\tableofcontents

\section{Introduction}

    Floer homotopy theory is the programme which seeks, given any particular Floer homology ${\it HF_*}$, to naturally associate to it a stable homotopy type $\it SF$ called the \emph{Floer homotopy type}, whose singular homology $H_*({\it SF})$ recovers the original ${\it HF}_*$. A very concrete solution in this regard, introduced in \cite{CJS} by Cohen, Jones, and Segal, requires as input the following data:
    \begin{itemize}
        \item[\sc i.] A \emph{flow category}, i.e. compact moduli spaces $\scr M_{i,j}$ of unparameterized Floer trajectories of all dimensions, equipped with smooth corner structures, such that boundary strata encode breaking at intermediary critical points: $\partial^u \scr M_{i,j} \cong \scr M_{i,u} \times \scr M_{u,j}$;
        \item[\sc ii.] A \emph{stable normal framing}, i.e. suitably compatible embeddings of $\scr M_{i,j}$ into Euclidean spaces with corners $\bb E_{i,j}$, endowed with compatible trivializations of their normal bundles $N_{i,j}$. 
    \end{itemize}
    More detailed definitions will be given in {\sc\S\ref{sec:def}}. This method has already been fruitful in the construction of the Khovanov homotopy type \cite{LS}. Other approaches to Floer homotopy theory are known, including Manolescu's use of the Conley index \cite{ManSWF}, Abouzaid and Blumberg's spectrum-valued chains \cite{ABMorava} or their more recent stable $\infty$-categorical approach \cite{AB}, as well as variations on \cite{CJS}, e.g. Large \cite{Large} or C\^ot\'e and Kartal \cite{CoteKartal}. We do not address these other methods here.

    The most natural place to test any particular approach to Floer homotopy theory is Morse theory. Here, the natural candidate for the Floer homotopy type is $M$ itself, or rather the suspension spectrum $\Sigma^\infty_+ M$. To carry out the Cohen-Jones-Segal (CJS) construction in the setting of Morse theory, one needs smooth moduli spaces of broken flow lines, assembling into a flow category, e.g. those constructed in \cite{MorseBott} or \cite{BanHurt} in Morse-Bott generality, as well as stable normal framings thereof. It is the latter piece of data that is quite difficult to obtain rigorously; the CJS paper \cite{CJS} sketches a construction of these stable framings, by embedding the moduli space $\scr M_{i,j}$ of unparameterized flow lines from $p_i$ to $p_j$ into the unit sphere $S_i^-$ inside the negative subspace of the Hessian at $p_i$, and showing that $\scr M_{i,j}$ is coframed there, resulting in a stable trivialization of $T\scr M_{i,j}$. Sadly, this only stably frames the \emph{interior} of the compactified moduli space, as the map $\scr M_{i,j} \to S_i^-$ fails to be an embedding as soon as breaking occurs. The authors of \cite{CJS} address this only in passing, asserting that the issue can be remedied by a ``canonical perturbation.'' The more recent paper of Abouzaid and Blumberg \cite{ABMorava} does construct \emph{tangential} framings on the compactified moduli spaces, but only for non-degenerate critical points, suitable for their setup; translating them into the CJS language seems difficult. Rezchikov \cite{Rez} also gives a rigorous construction of stable framings, but using a functional-analytic method instead, and only for non-degenerate critical points.
    
    Lastly, if the CJS construction is the right approach to Floer homotopy theory, one would expect it to recover the stable homotopy type of $M$, when applied to the flow category gotten via Morse theory. While the original CJS paper \cite{CJS} hints at a proof of this fact, by filtering $M$ according to its sub-level sets, they do not provide a rigorous proof. Abouzaid and Blumberg \cite{ABMorava} construct such an isomorphism in their setup of spectrum-valued chains, and at least their construction of the map is vaguely similar to ours, although their proof that it is an isomorphism differs substantially. In fact, our idea closely resembles the more recent proof offered by C\^ot\'e and Kartal in \cite{CoteKartal} in their slightly different setup. The map is constructed using compactified moduli spaces $\scr W_i$ of unparameterized broken trajectories with free endpoint. To our knowledge, neither \cite{ABMorava} nor \cite{CoteKartal} provide a construction of these spaces as smooth manifolds with corners. In the present paper we also give an account of how to construct this smooth structure, in the generality of pseudo-gradient Morse-Bott flows.

    Given this background, the aim of the present paper is twofold. The first aim is to fill in the missing details we mentioned. In fact, following C\^ot\'e and Kartal \cite[Thm. 3.11]{CoteKartal}, we prove that one can recover every Thom spectrum on $M$, and not just $\Sigma^\infty_+M$. In addition to \cite{CoteKartal}, we further show that Thom spectra are \emph{all} the possible stable homotopy types arising from the CJS construction. Our results can be summarized as follows:
    \begin{UNI}[Main Theorem] \sl On a closed smooth manifold $M$, the following hold:
        \vspace{-0.4em}
        \begin{enumerate}
            \item[\sc i.] Any generic pseudo-gradient Morse-Bott flow $(M,f,\Phi)$ gives rise to a flow category $\scr F_{\rm MB} = \scr F_{\rm MB}(M, f, \Phi)$. {\sc(Thm.\,\ref{thm:sm})}
            \item[\sc ii.] For every reduced $KO$-theory class $[E] \in \widetilde{KO}(M)$, there exist stable normal framings called $E$-twisted framings {\sc(Def.\,\ref{def:can})}, for which the Cohen-Jones-Segal (CJS) construction on $\scr F_{\rm MB}$ recovers the Thom spectrum $M^{E} := \Sigma^{-{\rm rk}\; E} {\rm Th}(E)$. These framings can be chosen canonically up to a contractible parameter-space. {\sc(Thms.\,\ref{thm:fr}\,\&\,\ref{thm:main})}
            \item[\sc iii.] Any stable normal framing of $\scr F_{\rm MB}$ becomes $E$-twisted for some $[E] \in \widetilde{KO}(M)$ after suitable stabilization; therefore, Thom spectra are the only possible stable homotopy types that can arise from the CJS construction on $\scr F_{\rm MB}$. {\sc(Thm.\,\ref{thm:ofr})}
        \end{enumerate}
    \end{UNI}
    Part {\sc i} is well established in the literature \cite{MorseBott, BanHurt}, but we use a different technique, suitable for extending it to the moduli spaces $\scr W_i$ much needed in constructing the isomorphism of part {\sc ii}. Although our construction of the isomorphism resembles that of \cite{CoteKartal}, their setup is quite different, as it is based on a reinterpretation of the CJS method via Atiyah duals, originally due to \cite{Large}. The rigorous construction of the $E$-twisted framings in {\sc Thm.\,\ref{thm:fr}} is undoubtedly our most delicate result, and to our knowledge does not appear anywhere else in this degree of detail (cf. \cite[\S D.]{ABMorava}). Part {\sc iii} also appears to be completely new.

    When applying part {\sc ii} of our main result for $[E] = 0$, the CJS construction recovers $\Sigma^\infty_+ M$, as expected. Of particular interest is also $[E] = -[TM]$, in which case the CJS construction recovers the \emph{Spanier-Whitehead dual} of $M$, by \cite{AtTh}, up to a shift; so that the Floer homotopy type $\it SF$ now satisfies ${\it HF_*} \cong H^{-*}({\it SF})$, instead of $H_*({\it SF})$. Either approach may be desirable, since the cohomological grading convention is widely used in Floer theory.

    The second aim of our exposition is to provide the necessary background for a future paper, treating the Floer homotopy theory of monotone Lagrangians. Since the moduli spaces fail to be manifolds as soon as their dimension reaches $N_L$, the minimal Maslov number, and since even the lower-dimensional moduli spaces are not always frameable, we only get a module over the Postnikov truncation $\tau_{\le (N_L - 2)}R$ for some bordism spectrum $R$. In order to give interesting examples of this, we will need to construct stable-homotopy lifts of the Piunikhin-Salamon-Schwarz (PSS) continuation maps of Albers \cite{Albers} in Lagrangian Floer theory. Hence, in the present paper, we give an account of continuation maps, and their functoriality, taking \cite{AB} for inspiration. We explain the differential topology behind PSS maps with no symplectic geometry involved, and also show how to relax the condition of stable framing to orientability with respect to orthogonal spectra, inspired by Cohen's subsequent paper \cite{Cohen}.

    \begin{RMK}
        In our main result stated above, we have assumed that $M$ is a closed finite-dimensional manifold, for the sake of simplicity. The proofs of {\sc i} and {\sc ii} can be extended easily (at least for $E = 0$) to the case that $M$ is still finite-dimensional but potentially non-compact, with $f$ being proper and bounded below. The only difference is that now one needs to allow flow categories to be unbounded above; the appropriate notion of Cohen-Jones-Segal construction required for this setup is defined by taking a direct limit of CJS constructions for the finite stages. Dealing with general $[E] \in \widetilde{KO}(M)$ is probably feasible, but requires some care in defining virtual bundles over non-compact spaces. If one further allows $f$ to be unbounded below as well, more difficulties arise: for example, the CJS construction would additionally involve a homotopy inverse limit (which in general does not commute with either $\Sigma^\infty$ or $H_*$), or the use of pro-spectra \cite[Appendix]{CJS}; also, some of our inductive arguments would cease to work. We expect the CJS construction to recover some sort of completion of $\Sigma^\infty_+ M$ with respect to the sub-level set filtration induced by $f$.
        
        Finally, one may wonder if the results extend to the case that $M$ is an infinite dimensional Hilbert manifold, with $f$ bounded below and satisfying the Palais-Smale condition \cite{PS}. Our work in {\sc\S\ref{sec:stfr}} of constructing stable normal framings, as well as the construction of the isomorphism in {\sc\S\ref{sec:cjs}} would likely adapt without serious changes. The only part that perhaps requires more subtlety is our construction of the smooth structure on $\scr M_{i,j}$ and $\scr W_i$ in {\sc\S\ref{sec:smooth}}, since that would require us to define blow-ups for infinite-dimensional manifolds. It seems that an infinite-dimensional version of our result would be of some interest in symplectic geometry, e.g. for the free loop-space on a manifold, see Cohen's \cite[\sc\S3]{CohenCotangent}, but in the present paper we limit ourselves to treating only the closed, finite-dimensional case.
    \end{RMK}

    \subsection{Organization and Conventions} In {\sc\S\ref{sec:def}}, we review smooth manifolds with corners, along the lines of \cite{Joyce}, which form the essential building block for running the CJS construction. Our notion of smooth manifold with $I$-corners in {\sc Def.\,\ref{def:corn}} is formulated in a new way, but agrees with the notion of $\<n\>$-manifold used in other Floer homotopy theory papers, e.g. \cite[{\sc\S 5}]{ManSar}, when $I = \{1, \ldots, n\}$. We also introduce the \emph{oriented blow-up} of smooth manifolds with corners, along properly embedded submanifolds, which appears to be related to the notion of ``real blow-up'' popular in geometric analysis, cf. \cite{Melrose} and \cite{KotMel}, and also the ``spherical blow-up'' of \cite{ArKa}. This will be an important tool in constructing the smooth moduli spaces associated to Morse-Bott theory. Finally, we also give the relevant definitions of flow categories, embeddings and stable normal framings thereof, in the more general Morse-Bott-like setup, allowing critical points to become critical manifolds.\pagebreak[3]
    
    Our indexing convention orders the critical sets in \emph{descending}, rather than the more standard ascending order of their $f$-value, cf. \cite{CJS}. This seems more appropriate given the nature of the recursive definitions in {\sc\S\ref{sec:smooth}}, and the direction of the continuation maps in {\sc\S\ref{sec:pss}}.

    In {\sc\S\ref{sec:smooth}}, we undertake the rigorous construction of smooth structures with corners on the various moduli spaces $\scr M_{i,j}$ associated to a generic Morse-Bott flow, by means of a recursive definition involving oriented blow-ups, thereby producing a flow category. Constructing such $\scr M_{i,j}$ has been done before through different techniques, e.g. in \cite{BanHurt} (see also \cite{BurHal}, \cite{Weh} or \cite{Zhou}). However, our construction additionally includes that of moduli spaces $\scr W_i$ of flow lines with free endpoint, which those references do not treat. Abouzaid-Blumberg \cite[\sc\S D]{ABMorava} mention manifolds ``$\scr M(x, M)$'', which are likely the same as our $\scr W_i$ in the Morse case, although they do not explain how to endow them with smooth corner structures; likewise for C\^ot\'e-Kartal's $\scr N(i)$ in their \cite[\sc Def.\,3.15]{CoteKartal}. We also point out that the genericity condition we will employ appears to be weaker than what is usually termed ``Morse-Bott-Smale,'' and it appears to be symmetric with respect to reversing the flow (whereas Morse-Bott-Smale is not), see {\sc Rmk.\,\ref{rmk:sym}}.

    In {\sc\S\ref{sec:stfr}}, we give a rigorous construction of the $E$-twisted stable normal framings necessary to carry out the CJS construction for the flow category associated to a generic Morse-Bott flow. Although the author has tried to use the suggested method in \cite{CJS}, of ``canonically perturbing'' the maps $\scr M_{i,j} \to S_i^-$ to become embeddings (which in fact motivated our use of oriented blow-ups, in the hope for such a ``canonical'' way), the process seems too cumbersome. Instead, we use a completely different approach, more similar to the tangential version produced by Abouzaid and Blumberg \cite{ABMorava}, namely to first frame the $\scr W_i$ in a suitable way, and then obtain the framings of the $\scr M_{i,j}$ as a result. We also treat the question of whether any set of stable normal framings is $E$-twisted for some $E$, in {\sc Thm.\,\ref{thm:ofr}}.

    In {\sc\S\ref{sec:cjs}}, we give an abstract definition of the Cohen-Jones-Segal construction, and also provide a rigorous proof that in the case of Morse-Bott theory, it recovers the Thom spectrum $M^E$ from the $E$-twisted framings. Although for simplicity we only explain the CJS construction for a \emph{finite} flow category (see \cite{CJS} for how to treat the infinite case,) our version is in a way more general than those presented in other places, e.g. \cite{ManSar} and \cite{LS}, in that it does not require one to group all critical points of the same Maslov index together. In fact, in the Morse-Bott setting this would not even make sense. We additionally prove the stable homotopy-invariance of the CJS construction under homotopies of the input data, and show that the cellular chain complex of the CJS construction recovers the abstract Morse complex associated to any classical flow category, i.e. all of whose critical loci are points.

    Finally, in {\sc\S\ref{sec:pss}} and {\sc\S\ref{sec:or}}, we give the necessary technical prerequisites for building PSS-style maps. In {\sc\S\ref{sec:pss}}, we show how to make the CJS construction functorial in a certain sense, construct exact triangles, and explain the abstract differential topology behind upgrading PSS maps to maps in the stable homotopy category. The approach appears to be quite different from that of Cohen \cite{CohenCotangent}. We also point out how Floer homotopy theory actually allows one to considerably simplify the definition of a PSS map, in contrast to ordinary Floer homology. Lastly, {\sc\S\ref{sec:or}} discusses what one should do if stable normal framings do not exist, by instead orienting the normal bundles $N_{i,j}$ relative to orthogonal spectra $R_{i,j}$; the ideas are not new, cf. \cite{Cohen}, but we include them here for the sake of completeness.

    \subsection{Acknowledgements} I would very much like to thank my advisor, Ciprian Manolescu, for guiding me and offering many useful suggestions throughout the writing process, and also to Mohammed Abouzaid for suggesting the key idea in constructing the framings of {\sc\S\ref{sec:stfr}}, namely to characterize them by their extendability over the $\scr W_i$, along the lines of \cite{ABMorava}. I am also very grateful to Laurent C\^ot\'e and Bari\c s Kartal for informing me about their paper \cite{CoteKartal}, and for helping me understand the differences between our approaches. The exact triangle approach taken in {\sc\S\ref{sec:pss}} for continuation maps was inspired by a discussion with Semon Rezchikov. The interpretation of {\sc Ex.\,\ref{exp:twist}} via twisted spectra was brought to my attention by Alice Hedenlund and Trygve Oldervoll. I am also greatly indebted to Stanford University, and to my sponsors for the William R. and Sara Hart Kimball Fellowship.

\section{Preliminaries of differential topology and definitions}\label{sec:def}

    In this section we recall the definitions of smooth manifolds with corners, see \cite{Joyce} (and also the more historical \cite[\sc\S 4]{Jan}, \cite{Douady}); and oriented blow-ups thereof, cf. \cite{ArKa}, which will be used as an integral part of the constructions in {\sc\S\ref{sec:smooth}}. We also recall the notion of a flow category, inspired by \cite{LS}, \cite{ManSar}, generalized to the Morse-Bott setup, together with a few other notions attached to it.

    \begin{DEF}
        Given a set $I$, we define an $I$-stratification on a set $X$ to be a collection of subsets $\{\partial^i X \subseteq X\}_{i \in I}$. Given a subset $J \subset I$, we define the (closed) $\partial^J$-boundary stratum of $X$ as the intersection $\partial^J X := \bigcap_{j \in J} \partial^j X.$ An {$I$-map} $f : X \to Y$ between two $I$-stratified sets is a map satisfying $f^{-1}(\partial^i Y) = \partial^i X$ for all indices $i \in I$.
    \end{DEF}

    \begin{EXP}\label{exp:Istr}
        The space $\bb R_{\ge 0}^I$ of $I$-tuples of non-negative real numbers is endowed with the $I$-stratification for which $\partial^i\bb R_{\ge 0}^I$ consists of the tuples whose $i$-entry is zero. Also, whenever $X$ and $Y$ are $I$ and $J$-stratified respectively, we implicitly understand $X \times Y$ to be stratified by $I \sqcup J$, where $\partial^i(X \times Y) = \partial^i X \times Y$, and $\partial^j(X \times Y) = X \times \partial^j Y$ for all $i \in I, j \in J$. Any space is considered to be $\0$-stratified, unless otherwise stated. If $X$ is $I$-stratified, and $I' \supset I$, we implicitly understand $X$ to be $I'$-stratified with $\partial^i X = \0$ for $i \in I' \setminus I$. If $X$ is $I$-stratified and $Y \subset X$, we implicitly understand $Y$ to be $I$-stratified via $\partial^i Y := Y \cap \partial^i X$.
    \end{EXP}

    \begin{DEF}\label{def:corn}
        An $I$-atlas of dimension $n$ on an $I$-stratified topological space $X$ is defined to be a collection $\{(U_\lambda, V_\lambda, \varphi_\lambda)\}_{\lambda \in \Lambda}$ of triples where
        \begin{itemize}
            \item[\sc i.] $U_\lambda \subset X$ are open subsets covering $X$;
            \item[\sc ii.] $V_\lambda \subset \bb R^{n-|J_\lambda|} \times \bb R_{\ge 0}^{J_\lambda}$ are open subsets, for a finite subset $J_\lambda \subseteq I$ potentially depending on $\lambda \in \Lambda$;
            \item[\sc iii.] $\varphi_\lambda : U_\lambda \overset\sim\to V_\lambda$ an $I$-homeomorphism, i.e. a homeomorphism which is also an $I$-map;  
        \end{itemize}
        such that the transition maps $\varphi_\mu \circ \varphi_\lambda^{-1}$ are smooth, meaning that there is an extension thereof to a smooth function on a neighborhood of the domain in Euclidean space. A \keywd{smooth structure with $I$-corners} on $X$ is a maximal atlas, or an equivalence class of atlases, following the usual definition. Smooth maps are likewise defined in the usual manner.
    \end{DEF}

    \begin{DEF}
        Let $V$ be a vector space. Consider the oriented projective space $\bb P^+(V) := (V - \{0\})/\bb R_{>0}$ defined as the quotient of the action of $\bb R_{>0}$ by multiplication. This carries an $\bb R_{\ge 0}$-bundle called the tautological half-line bundle $\hat V$, defined as the set of pairs $([w], v)$ in $\bb P^+(V) \times V$ satisfying $v = \lambda w$ for some $\lambda \ge 0$. These two notions extend in the obvious way to any vector bundle $E$, i.e. we can talk about the oriented projectivization $\bb P^+(E)$ and the tautological half-line bundle $\hat E$ on it, which will be used as local model for the oriented blow-up. Note that the forgetful map $\pi : \hat E \to E$ given by $\pi([w],v) = v$ induces an isomorphism $\hat E - \{\text{0-sec}\} \cong E - \{\text{0-sec}\}$.
    \end{DEF}

    \begin{DEF}\label{def:bu}
        Let $X$ be a smooth manifold with $I$-corners, and $V \subset X$ be a smooth, \emph{properly} embedded $I$-submanifold, i.e. such that $V \hookrightarrow X$ is an $I$-map transverse to all the lower-dimensional boundary strata of $X$. We define the \keywd{oriented blow-up} ${\rm Bl}^+_V X$ (or $\hat X$ for brevity) of $X$ along $V$, as a set, to be the disjoint union of $X - V$ and the oriented projectivization $\bb P^+(\scr N)$ of the normal bundle $\scr N = TX/TV$. We endow the blow-up with an $I \sqcup \{\beta\}$-stratification: the $\beta$-boundary is defined to be $\bb P^+(\scr N)$, also called the exceptional divisor; for the other $i \in I$, we define $\partial^i\hat X := \pi^{-1}(\partial^i X)$, where $\pi : \hat X \twoheadrightarrow X$ is the \emph{blow-down} which maps $X - V$ via the identity inside $X$, and $\bb P^+(\scr N)$ via the projection onto $V$.
        
        To endow this with the structure of a smooth manifold with $I \sqcup \{\beta\}$-corners, we choose a tubular neighborhood $\phi : U \hookrightarrow X$, where $U$ is some open neighborhood of the zero-section in the normal bundle $\scr N = TX/TV$ over $V$, and present ${\rm Bl}^+_V X$ as the union of $X - V$ and $\hat U := \pi^{-1}(U) \subset \hat{\scr N}$ over the open subset $U^\circ = U - \{\text{0-sec}\}$, under the map $\phi|_{U^\circ}$. (The submanifold $V \subset X$ must be \emph{properly} embedded, in order to have a good notion of tubular neighborhood; see \cite[\sc \S II.5]{Douady}.)
    \end{DEF}

    \begin{PROP}[Blow-up is well-defined]
        Any two choices of tubular neighborhoods $\phi, \phi'$ give equivalent smooth structures with $I \sqcup \{\beta\}$-corners on ${\rm Bl}^+_V X$.
    \end{PROP}

    \begin{proof}
        Put differently, the conclusion asserts that the identity map from the blow up to itself is smooth, where the smooth structure on the domain is inherited from a tubular neighborhood $\phi : U \hookrightarrow X$, and the smooth structure on the target is inherited from another tubular neighborhood $\phi' : U' \hookrightarrow X$. This assertion is clear on $X - V$, since the smooth structure there is pulled back from $X$, and its definition does not even mention $\phi, \phi'$. Hence, the non-trivial part consists in showing smoothness of the identity at a point in $\bb P^+(\scr N)$.

        It suffices to show that the transition map $\hat \psi$ on $\hat{\scr N}$ induced by the locally defined diffeomorphism $\psi := \phi^{-1} \circ \phi'$ near the zero-section in $\scr N$ is smooth. Let us use local coordinates $\vec x = (x_1, \ldots, x_n) \in \bb R^n$ for the normal directions, $\vec v$ for the $V$-direction, and write $\vec x = r \hat x$ with $\hat x \in \bb S(\bb R^n) \cong \bb P^+(\bb R^n)$ of unit modulus. Further, let us split $\psi = \psi^\perp + \psi^\parallel$ according to the normal and $V$-components of $\psi$ (on the codomain of $\psi$), so in particular $\psi^\perp|_V = 0$ and $\psi^\parallel|_V = {\rm id}_V$. Now, in $(r, \hat x, \vec v)$-coordinates, $\hat \psi$ by definition is
        \[ \hat \psi(r, \hat x, \vec v) = \begin{cases}
            \left(|\psi^\perp|, \frac{\psi^\perp}{|\psi^\perp|}, \psi^\parallel\right) & \text{if } r > 0, \\
            \left(0, \frac{d\psi^\perp(\hat x)}{|d\psi^\perp(\hat x)|}, \vec v\right) &\text{if } r = 0.
        \end{cases} \]
        (Denominators are non-zero because $\psi$ is a diffeomorphism.) By Taylor's theorem, we write
        \[ \psi^\perp = x_1 F_1 + \cdots + x_n F_n = r (\hat x_1 F_1 + \cdots + \hat x_n F_n), \]
        in virtue of the fact that $\psi^\perp|_V = 0$ where $F_i$ are again smooth functions. With this in mind, both branches of the equation can be rewritten as one united formula \pagebreak[2]
        \[ \hat \psi(r, \hat x, \vec v) = \left(r |\hat x_1 F_1 + \cdots + \hat x_n F_n|, \frac{\hat x_1 F_1 + \cdots + \hat x_n F_n}{|\hat x_1 F_1 + \cdots + \hat x_n F_n|}, \psi^\parallel\right). \]
        This is now visibly smooth, since the absolute value is always evaluated away from 0. 
    \end{proof}

    \begin{PROP}[Fiber-products of manifolds with corners]\label{prop:fib}
        Let $Z$ be a smooth corner-less manifold, and let $X$ and $Y$ be smooth manifolds with $I$ and $J$-corners respectively. If $f : X \to Z$ and $g : Y \to Z$ are smooth maps, such that $f|_{\partial^{I'}} \pitchfork g|_{\partial^{J'}}$ for all $I' \subseteq I, J' \subseteq J$, then $X \times_Z Y$ is naturally a smooth manifold with $(I \sqcup J)$-corners.
    \end{PROP}

    \begin{proof}
        This is a special case of \cite[\sc\S 6]{Joyce}, in light of \cite[\sc Rmk. 3.3]{Joyce}.
    \end{proof}

    \begin{PROP}[Blow-ups and fiber-products]\label{prop:blfib}
        Let $V \subset X$ be an embedding of smooth corner-less manifolds, and $f : Y \to X$ be a smooth map from a smooth manifold $Y$ with $I$-corners, such that all lower strata of $Y$ are transverse $V$. Then, $W := V \times_X Y \subset Y$ is a properly embedded $I$-submanifold, and the map $f : Y \setminus W \to X \setminus V$ extends uniquely to a smooth map of blow-ups ${\rm Bl}^+ f : {\rm Bl}^+_{W} Y \to {\rm Bl}^+_V X$.
    \end{PROP}

    \begin{proof}
        That $W = V \times_X Y \subset Y$ is a properly embedded submanifold follows from {\sc Prop.\,\ref{prop:fib}}, and some linear algebra. It is clear how to extend the map in question set-theoretically, since the normal bundle of $W \subset Y$ is pulled back from that of $V \subset X$, and hence there is a map relating their projectivizations. To check that this map is smooth, it suffices to pick tubular neighborhoods $N_{W \subset Y} \supset U_W \overset \phi\to Y$ and $\psi : \scr N_{V \subset X} \supset U_V \overset\psi\to X$ which respect the map $f$, and note that the map ${\rm Bl}^+ f$ in the local coordinates induced by these two tubular neighborhoods is simply the induced map $\widehat{\scr N}_{W \subset Y} \to \widehat{\scr N}_{V \subset X}$ on tautological half-line bundles, which is smooth. The existence of tubular neighborhoods $\phi, \psi$ compatible with $f$ can be achieved locally by putting the span $V \hookrightarrow X \leftarrow Y$ in a standard position via the implicit function theorem, see \cite[\sc Prop. 5.1]{Joyce}, and reducing to a linear algebra exercise. See also \cite{ArKa}, proving a similar proposition for their ``spherical blow-up'' without corners.
    \end{proof}

    We now introduce the notions required for the CJS construction of {\sc\S\ref{sec:cjs}}. The following definition is motivated by the limiting behavior of pseudo-gradient flow lines of a generic Morse-Bott flow (see {\sc \S\ref{sec:smooth}}), or a Floer functional on an infinite-dimensional configuration space. We implicitly allow manifolds to have components of different dimensions, so $\dim X$ tacitly denotes a locally constant function on any manifold $X$.

    \begin{DEF}\label{def:flow}
        A \keywd{flow category} $\scr F$ consists of a finite, totally ordered set $I$ of indices, a collection of closed smooth manifolds $\{\scr C_i\}_{i \in I}$ referred to as \emph{critical manifolds}, and a collection of compact smooth manifolds $\scr M_{i,j}$ with $\{i < u < j\}$-corners for all $i < j$ in $I$, equipped with the following extra data. There are smooth \emph{source} and \emph{target} maps $s_{i,j} : \scr M_{i,j} \to \scr C_i$ and $t_{i,j} : \scr M_{i,j} \to \scr C_j$, which satisfy all the transversality conditions (see {\sc Prop.\,\ref{prop:fib}}) needed to ensure that the iterated fiber product
        \[\numeq\label{eq:multifib}
            \scr M_{i_0,i_1} \times_{\scr C_{i_1}} \scr M_{i_1,i_2} \times_{\scr C_{i_2}} \cdots \times_{\scr C_{i_{n-1}}} \scr M_{i_{n-1},i_n},
        \]
        along source and target maps, is a smooth manifold with corners for any sequence of indices $i_0 < i_1 < \cdots < i_n$ in $I$. Lastly, we require identifications of \ref{eq:multifib} with $\partial^{\{i_1, \ldots, i_{n-1}\}}\scr M_{i_0,i_n}$ which are associative in the usual sense (so that one obtains a non-unital category with objects $I$.)
    \end{DEF}

    \begin{RMK}
        Standard linear algebra shows that the transversality condition for \ref{eq:multifib} is independent of the order in which we take the fiber-products, and means that at the tangent space level, all fiber-products have the expected dimension.
    \end{RMK}

    \begin{CONV}\label{conv:ind}
        If all the $\scr C_i$ are points, we say that $\scr F$ is \emph{classical}. For legibility, we write $\times_i$ instead of $\times_{\scr C_i}$ for fiber-products. We also often pull back bundles from $\scr C_i$ or $\scr C_j$ to $\scr M_{i,j}$ without writing $s_{i,j}^*$ or $t_{i,j}^*$.
    \end{CONV}

    \begin{RMK}
        Note that the boundary identifications in the definition of a flow category impose the following condition (at least on components of $\scr M_{i,j}$ having $\partial^u \neq \0$):
        \[ \dim \scr M_{i,j} = 1 + \dim \scr M_{i,u} + \dim \scr M_{u,j} - \dim \scr C_u, \]
        as locally constant functions. Equivalently, the quantity $1 + \dim \scr M_{i,j} - \dim \scr C_i$ is additive under composition. A \emph{Maslov grading}, making this relative quantity absolute, consists of locally constant functions $\mu(i) : \scr C_i \to \bb Z$, $\forall i \in I$, satisfying $1 + \dim \scr M_{i,j} - \dim \scr C_i = \mu(i) - \mu(j)$. A \keywd{graded flow category} is a flow category $\scr F$, equipped with a Maslov grading.
    \end{RMK}

    \begin{DEF}
        Given a finite set $I$, we use the term \emph{Euclidean $I$-space} to denote any space of the form $V \times \bb R_{\ge 0}^I$ for $V$ a vector space, with the $I$-stratification from {\sc Ex.\,\ref{exp:Istr}}.
    \end{DEF}

    \begin{DEF}\label{def:cohemb}
        A \keywd{coherent collection of embeddings} for a flow category $\scr F$ as above, is given by Euclidean $\{i < u < j\}$-spaces $\bb E_{i,j}$ with associative identifications
        \[ \bb E_{i,j} = \bb E_{i,u} \times \bb R_{\ge 0}^{\{u\}} \times \bb E_{u,j}, \]
        together with proper $\{i < u < j\}$-embeddings $\iota_{i,j} : \scr M_{i,j} \hookrightarrow \scr C_i \times \bb E_{i,j}$ whose $\scr C_i$-entry is the source map $s_{i,j} : \scr M_{i,j} \to \scr C_i$, such that the boundary restriction $\iota_{i,j}|_{\partial^u}$ agrees with the map
        \begin{equation*}\numeq\label{eq:fibmap}
            \partial^u\scr M_{i,j} = \scr M_{i,u} \times_u \scr M_{u,j} \to \scr C_i \times \bb E_{i,u} \times \{0\} \times \bb E_{u,j}
        \end{equation*}
        given by the embedding $\iota_{i,u}$ into $\scr C_i \times \bb E_{i,u}$ on the first factor, and by the embedding $\iota_{u,j}$, composed with the projection onto $\bb E_{u,j}$, on the second factor.
    \end{DEF}

    \begin{RMK}\label{rmk:normadd}
        The map \ref{eq:fibmap} is an injection. Indeed, the $\scr C_i \times \bb E_{i,u}$ factor uniquely identifies the $\scr M_{i,u}$-entry, which then specifies the $\scr C_u$-source of the $\scr M_{u,j}$-entry, which together with $\bb E_{u,j}$ uniquely determines the pair completely. The same argument applied at the tangent space level shows that the map \ref{eq:fibmap} is actually a smooth embedding. A modified version of this argument also shows that any two choices of normal bundles $N_{i,u}$ and $N_{u,j}$ of the two embeddings $\iota_{i,u}$ and $\iota_{u,j}$, i.e.
        \[ T\scr M_{i,u} \oplus N_{i,u} \overset\sim\to T\scr C_i \oplus T\bb E_{i,u}, \qquad T\scr M_{u,j} \oplus N_{u,j} \overset\sim\to T\scr C_u \oplus T\bb E_{u,j} \]
        add up to a choice of normal bundle $N_{i,j}|_{\partial^u} \cong N_{i,u} \oplus N_{u,j}$ for $\iota_{i,j}$:
        \[ T(\scr M_{i,u} \underset{u}\times \scr M_{u,j}) \oplus N_{i,u} \oplus N_{u,j} \overset\sim\to T\scr C_i \oplus T\bb E_{i,u} \oplus T\bb E_{u,j}. \]
    \end{RMK}

    \begin{DEF}\label{def:snf}
        Let $\xi_i$ be vector bundles on the critical manifolds $\scr C_i$ of a graded flow category $\scr F$, such that the quantity ${\rm rk\;} \xi_i - \mu(i)$ is constant as a function on $\scr C_i$. We define a \keywd{stable normal framing of $\scr F$ relative to $\xi_i$} to be a coherent collection of embeddings $\iota_{i,j} : \scr M_{i,j} \hookrightarrow \scr C_i \times \bb E_{i,j}$ with normal bundles $N_{i,j}$ as in the remark above, a collection of vector spaces $V_{i,j}$ with associative identifications $V_{i,u} \oplus V_{u,j} = V_{i,j}$, and finally vector bundle isomorphisms $\phi_{i,j} : \xi_i \oplus N_{i,j} \overset\sim\to V_{i,j} \oplus \xi_j$ over $\scr M_{i,j}$ such that the boundary restriction $\phi_{i,j}|_{\partial^u}$ is equal to the composite
        \[\numeq\label{eq:snf} \xi_i \oplus N_{i,j} \cong \xi_i \oplus N_{i,u} \oplus N_{u,j} \overset{\phi_{i,u}}{\longrightarrow} V_{i,u} \oplus \xi_u \oplus N_{u,j} \overset{\phi_{u,j}}{\longrightarrow} V_{i,u} \oplus V_{u,j} \oplus \xi_j \cong V_{i,j} \oplus \xi_j. \]
    \end{DEF}

\section{Morse-Bott moduli spaces and their smooth structures}\label{sec:smooth}

    In this section, we provide examples of flow categories via Morse-Bott theory. Consider a closed smooth manifold $M$, with a \keywd{Morse-Bott function} $f : M \to \bb R$, i.e. a smooth map whose Hessian $H(-,-)$ has locally constant nullity when restricted to the critical locus $\scr C := {\rm Crit}(f)$. This automatically implies that $\scr C$ is a smooth manifold of dimension ${\rm null}(H_p)$ at any point $p \in \scr C$, by the constant-rank theorem.

    Let $r_1 > \cdots > r_n$ denote the critical values of $f$, in decreasing order, and $\scr C_i := \scr C \cap f^{-1}(r_i)$. Note that the various $\scr C_i$ may still have several connected components, of potentially different dimensions. Now, at each point $p \in \scr C_i$, the tangent space $T_pM$ is canonically divided into a direct sum of three subspaces
    \[ T_pM = T_p^- \oplus T_p^0 \oplus T_p^+, \]
    namely the sum of all the eigenspaces corresponding to negative, zero, or positive eigenvalues respectively. These can be assembled into vector bundles $T_i^-, T_i^0, T_i^+$ over $\scr C_i$, since the projection operators onto the three subspaces can be written as continuous functions in the symmetric matrix representing the Hessian, in any given local coordinate system. The bundles $T_i^\pm$ are equipped with canonical Euclidean structures, via $\|x_\pm\| := \pm H(x_\pm, x_\pm)$.

    Notice that $T_i^0 \cong T\scr C_i$, and hence that the normal bundle of the inclusion $\scr C_i \subset M$ is isomorphic to $T_i^- \oplus T_i^+$. Hence, a neighborhood $\Omega_i$ of $\scr C_i$ can be identified with a neighborhood of the zero-section in $T_i^- \oplus T_i^+$. In fact, by the Morse-Bott lemma \cite{MBLem}, we can find local coordinates $(x_-, x_+) \in T_i^- \oplus T_i^+$ on $M$ near $\scr C_i$, such that
    \[\numeq\label{eq:coords}
        f(x_-, x_+) = r_i + \|x_+\|^2 - \|x_-\|^2, \qquad \forall x = (x_-, x_+) \in \Omega_i \subset T_i^- \oplus T_i^+.
    \]
    Let us choose a \keywd{pseudo-gradient} vector field $X$ on $M$, i.e. satisfying the two conditions
    \begin{itemize}
        \item[\sc i.] $df(X)$ is uniformly negative in the complement of all the $\Omega_i$.
        \item[\sc ii.] $X = -\tfrac 12 \nabla f = (x_-, -x_+)$ inside each $\Omega_i$.
    \end{itemize}
    The flow $\Phi_t$ along $X$ decreases $f$ strictly away from $\scr C$, and is the standard Anosov flow
    \begin{equation*}\numeq\label{eq:anosov}
        \Phi_t : (x_-, x_+) \mapsto (e^t x_-, e^{-t} x_+)
    \end{equation*}
    near the critical sets $\scr C_i$.

    \begin{DEFS}\label{def:sd}
        Some important notions immediately associated with the data $(M, f, \Phi)$ are the following:
        \begin{itemize}
            \item[\sc i.] For all $1 \le k < n$, we define the level-set $L_{k,k+1}$ as the preimage $f^{-1}(r)$ of any value $r \in (r_{k+1}, r_k)$. Any two choices of $r$ induce canonically isomorphic manifolds, via the flow $\Phi$. Similarly, we define the {open segment} $U_k := f^{-1}(a, b)$ for any $a \in (r_{k+1}, r_k)$, and any $b \in (r_k, r_{k-1})$, with $r_0 = \infty$ and $r_{n+1} = -\infty$ by convention. These $U_k$ can be arranged so that they cover $M$.
            \item[\sc ii.] There are ascending/descending open {disk} bundles $D_k^+, D_k^- \subset U_k$, defined as the loci of points flowing into/out of $\scr C_k$. Inside $\Omega_i$, they are canonically identified with open subsets of $T_k^\pm$. Similarly, there are {sphere} bundles $S_k^-, S_{k+1}^+ \subset L_{k,k+1}$, defined as the intersections of the formerly defined $D_k^-, D_{k+1}^+$ with $L_{k,k+1}$. They are canonically identified with the projectivizations $\bb P^+(T_k^-), \bb P^+(T_{k+1}^+)$.
        \end{itemize}
    \end{DEFS}

    \begin{DEFS}\label{def:moduli}
        We \emph{set-theoretically} define the following \keywd{compactified moduli spaces} of unparameterized broken flow lines:
        \begin{itemize}
            \item[\sc i.] For all $1 \le i < j \le n$, let $\scr M_{i,j}$ be the set of all potentially broken pseudo-gradient flow lines starting at $p_i$, and ending at $p_j$.
            \item[\sc ii.] For all $1 \le i < j \le n$, let $\scr V_{i,j}$ be the set of all potentially broken pseudo-gradient flow lines starting at $p_i$, and ending at any point in $L_{j-1,j}$.
            \item[\sc iii.] For all $1 \le i \le n$, let $\scr W_{i}$ be the set of all potentially broken pseudo-gradient flow lines starting at $p_i$, and ending at any point in $M$ (constant paths are allowed.)
        \end{itemize}
        We endow $\scr M_{i,j}, \scr V_{i,j}$ with $\{i < u < j\}$-stratifications where $\partial^u$ constitutes those flow lines passing through $p_u$. Similarly, we endow $\scr W_i$ with $\{i < u \le n\}$-stratifications where $\partial^u$ constitutes those flow lines passing through $p_u$ (which includes those ending at $p_u$.)
    \end{DEFS}

    \begin{RMK}
        It is quite clear that $\partial^u \scr M_{i,j} = \scr M_{i,u} \times_{\scr C_u} \scr M_{u,j}$, $\partial^u \scr V_{i,j} = \scr M_{i,u} \times_{\scr C_u} \scr V_{u,j}$, and $\partial^u \scr W_i = \scr M_{i,u} \times_{\scr C_u} \scr W_u$. These identifications are compatible for different $u$'s, in that concatenation of three or more broken flow lines does not depend on the order in which it is done. Following {\sc Conv.\,\ref{conv:ind}}, we write $\times_u$ instead of $\times_{\scr C_u}$.
    \end{RMK}

    These sets are more traditionally denoted with a bar on top, to distinguish them from the uncompactified versions. However, since our construction of these manifolds will produce the compactified moduli spaces directly, we have decided to drop the bar. Before we construct the smooth corner structures in {\sc Thm.\,\ref{thm:sm}}, we need a few preliminary results.

    \begin{RMK}\label{rmk:cof}
        All the disk and sphere bundles in {\sc Def.\,\ref{def:sd}-ii} are canonically co-framed relative to the $T_k^\pm$. That is, $TU_k/TD_k^\pm \cong T_k^\mp$ and consequently there are identifications $TL_{k-1,k}/TS_k^+ \cong T_k^-$ and $TL_{k,k+1}/TS_k^- \cong T_k^+$ over $\scr C_k$ via the target/source map.
    \end{RMK}

    \begin{PROP}[Surgery as blowing up and down]
        Any two adjacent level-sets $L_{k-1,k}$ and $L_{k,k+1}$ have canonically diffeomorphic oriented blow-ups
        \begin{equation*}\numeq\label{eq:samebl}
            {\rm Bl}^+_{S_k^+} L_{k-1,k} \cong {\rm Bl}^+_{S_k^-} L_{k,k+1}.
        \end{equation*}
        This common blow-up parameterizes potentially broken flow lines joining the two level-sets. The boundary is identified with $S_k^+ \times_k S_k^-$, the subspace of actually broken flow lines.
    \end{PROP}

    \begin{proof}
        As sets, it is clear that both blow-ups parameterize flow lines: away from the boundary exceptional divisor, the flow $\Phi$ induces a bijection between $L_{k-1,k} \setminus S_k^+$, $L_{k,k+1} \setminus S_k^-$, and the set of \emph{unbroken} flow lines. At the exceptional boundary divisors, the co-framings defined in {\sc Rmk.\,\ref{rmk:cof}} induce trivializations of the oriented projective bundles:
        \[ \bb P^+(TL_{k-1,k}/TS_k^+) \cong \bb P^+(T_k^-) \times_k S_k^+ \cong S_k^- \times_k S_k^+ \cong S_k^- \times_k \bb P^+(T_k^+) \cong \bb P^+(TL_{k,k+1}/TS_k^-). \]
        Putting these two together, we obtain a canonical bijection \ref{eq:samebl} of \emph{sets}. It remains to check that it is smooth in both directions, which we do by picking tubular neighborhoods of $S_k^+ \subset L_{k-1,k}$ and $S_k^- \subset L_{k,k+1}$. To understand the blow-ups, let us assume without loss of generality that the level-sets are cut out sufficiently close to the critical point itself, so as to enter into the standard neighborhood $\Omega_k$ of \ref{eq:coords}, and further assume that the level-sets are given by $f^{-1}(\pm 1)$ (which can be achieved by translation and rescaling of the original $f$.) We define tubular neighborhoods $T_k^- \times_k S_k^+ \to L_{k-1,k}$ and $S_k^- \times_k T_k^+ \to L_{k,k+1}$ by the formulas
        \[\numeq\label{eq:tubn} (x_-, \hat x_+) \mapsto (x_-, \sqrt{1 + \|x_-\|^2} \cdot \hat x_+), \qquad (\hat x_-, x_+) \mapsto (\sqrt{1 + \|x_+\|^2} \cdot \hat x_-, x_+), \]
        in coordinates \ref{eq:coords}. Both blow-ups are canonically diffeomorphic to $\bb R_{\ge 0} \times S_k^- \times_k S_k^+$ with blow-down projection maps to $L_{k-1,k}$ and $L_{k,k+1}$ given by
        \[ \pi^+(t, \hat x_-, \hat x_+) = (t \cdot \hat x_-, \sqrt{1 + t^2} \cdot \hat x_+) \quad {\rm and} \quad \pi^-(t, \hat x_-, \hat x_+) = (\sqrt{1 + t^2} \cdot \hat x_-, t \cdot x_+), \]
        respectively. It is not hard to see that $\pi^- = \Phi \circ \pi^+$, where $\Phi$ is the pseudo-gradient flow map from $L_{k-1,k}$ to $L_{k,k+1}$. Indeed, the flow is standard Anosov \ref{eq:anosov}, and so it keeps the unit-normalized directions of the two entries $\hat x_\pm$ the same, as well as the product of the sizes of these two vectors. This shows that our bijection is in fact the identity in the particular blow-up coordinates given by the tubular neighborhoods \ref{eq:tubn}, and hence smooth.
    \end{proof}

    \begin{PROP}\label{prop:backmap}
        The map $(U_k \setminus D_k^-) \to L_{k-1,k}$ which sends  points to their image under the backwards flow $\Phi^{-1}$ extends to a unique smooth map $\phi : {\rm Bl}^+_{D_k^-} U_k \to L_{k-1,k}$.
    \end{PROP}

    \begin{proof}
        Since $D_k^-$ is coframed in a canonical way by {\sc Rmk.\,\ref{rmk:cof}}, we have a canonical identification $\partial ({\rm Bl}^+_{D_k^-} U_k) \cong S_k^+ \times_k D_k^-$. We extend the map $\phi$ set-theoretically by letting it be the projection onto $S_k^+ \subset L_{k-1,k}$ at the exceptional divisor. It remains to check that this map is smooth at the exceptional divisor.

        Here, the analysis breaks into two cases. At points lying over $D_k^- \setminus \{0\text{-sec}\}$, the inclusion $D_k^- \subset U_k$ itself locally is canonically identified with $\bb R$ crossed with $S_k^- \subset L_{k,k+1}$, so that the map $\phi$ is smooth by \ref{eq:samebl}:
        \[ \smash{{\rm Bl}^+_{D_k^-} U_k \overset{\rm loc}\cong \bb R \times {\rm Bl}^+_{S_k^-} L_{k,k+1} \overset{\text{\ref{eq:samebl}}}\cong \bb R \times {\rm Bl}^+_{S_k^+} L_{k-1,k} \to L_{k-1,k}.} \]
        Near points lying over the zero-section $\scr C_k \subset D_k^-$, we use the local coordinates given by ${\rm Bl}^+_{D_k^-}U_k \cong D_k^- \times_k S_k^+ \times \bb R_{\ge 0}$, which we denote by $(x_-, \hat x_+, r)$. Direct calculation shows that the backwards flow map $\phi$ sends
        \[ (x_-, \hat x_+, r) \longmapsto (t \hat x_-, \sqrt{1 + t^2} \hat x_+) = \left(\frac r{\sqrt{1 + t^2}} x_-, \sqrt{1 + t^2} \hat x_+\right), \]
        where $t$ solves the equation $t \sqrt{1 + t^2} = \|x_-\| \cdot \|x_+\| = r \|x_-\|$. Indeed, this is because the flow preserves the unit-normalized vectors $\hat x_-, \hat x_+$, as well as the quantity $\|x_-\| \cdot \|x_+\|$. It only remains to see that this is smooth at $r = 0$. It suffices to show that $\sqrt{1 + t^2}$ is a smooth function of $(r, x_-)$. Indeed, squaring the defining equation for $t$ yields the equation
        \[ t^4 + t^2 - r^2 \cdot \|x_-\|^2 = 0. \]
        Among the two roots $t^2 = \tfrac 12\left(-1 \pm \sqrt{1 + 4r^2 \cdot \|x_-\|^2}\right)$, the solution must be the non-negative one, i.e. the one with plus. Hence,
        \[ \sqrt{1 + t^2} = \sqrt{\frac{1 + \sqrt{1 + 4r^2 \cdot \|x_-\|^2}}2}. \]
        This is smooth because $\|x_-\|^2$ is smooth in $x_-$, and the square-roots are always evaluated in the positive range of definition.
    \end{proof}

    \begin{THM}[Smooth upgrades of Defs.\,\ref{def:moduli} via oriented blow-ups]\label{thm:sm}
        After a generic perturbation of the vector field $X$ away from the $\scr C_i$, the following two formulas together give a recursive definition of smooth structures with corners on the sets $\scr M_{i,j} \subset \scr V_{i,j}$, together with smooth evaluation maps ${\rm ev}_{u,u+1} : \scr V_{i,j} \to L_{u,u+1}$ for every $i \le u < j$:
        \begin{itemize}
            \item[\sc i.] $\scr M_{i,j} := \scr V_{i,j} \underset{L_{j-1,j}}\times S_j^+$, where the first map $\scr V_{i,j} \to L_{j-1,j}$ is the evaluation map ${\rm ev}_{j-1,j}$, and the second one $S_j^+ \to L_{j-1,j}$ is the inclusion map.
            \item[\sc ii.] $\scr V_{i,i+1} := S_i^-$, with evaluation map ${\rm ev}_{i,i+1}$ the inclusion map; and $\scr V_{i,j+1} := {\rm Bl}^+_{\scr M_{i,j}} \scr V_{i,j}$ for $j > i$, with evaluation maps ${\rm ev}_{u,u+1}$ inherited from $\scr V_{i,j}$ for $i \le u < j$, and ${\rm ev}_{j,j+1}$ given by the composite ${\rm Bl}^+_{\scr M_{i,j}} \scr V_{i,j} \overset{\text{\sc Prop.\,\ref{prop:blfib}}}{-\!\!\!-\!\!\!\longrightarrow} {\rm Bl}^+_{S_j^+} L_{j-1,j} \stackrel{\text{\ref{eq:samebl}}}= {\rm Bl}^+_{S_{j}^-} L_{j,j+1} \to L_{j,j+1}$.  The new boundary coming from the blow-up is denoted by $\partial^{j}$.
        \end{itemize}
        Similarly, we define a smooth structure with corners on $\scr W_i$, with smooth evaluation map ${\rm ev} : \scr W_i \to M$ at the endpoint, by defining it on the open cover $\{{\rm ev}^{-1}U_j\}_{i \le j \le n}$:
        \[ {\rm ev}^{-1}U_j := \begin{cases}
            D_j^- & \text{if } j = i \\
            \scr V_{i,j} \underset{L_{j-1,j}}\times {\rm Bl}^+_{D_j^-} U_j & \text{if } j > i,
        \end{cases} \]
        where in the fiber product the map $\scr V_{i,j} \to L_{j-1,j}$ is ${\rm ev}_{j-1,j}$, and the map ${\rm Bl}^+_{D_j^-}U_j \to L_{j-1,j}$ is that of {\sc Prop.\,\ref{prop:backmap}.} At the $j^{\rm th}$ blow up, we denote the new boundary by $\partial^j$.
    \end{THM}

    \begin{proof}
        We must do two things: show that these definitions truly make sense after generic perturbation of $X$, and that they indeed agree set-theoretically with {\sc Defs.\,\ref{def:moduli}}.
        
        First, one must ensure the validity of the fiber-product $\scr V_{i,j} \underset{L_{j-1,j}}\times S_j^+$, which requires the sphere $S_j^+ \subset L_{j-1,j}$ to meet all the lower strata of $\scr V_{i,j}$ transversely (under the map ${\rm ev}_{j-1,j}$.) We claim that this can be achieved by a small generic perturbation of the pseudo-gradient $X$ between $\scr C_{j-1}$ and $\scr C_j$. First, by general transversality theory, it follows that for any generic small family of diffeomorphisms $\{\Psi_t\}_{t \in [0,\epsilon]}$ of $L_{j-1,j}$, with $\Psi_{0} = {\rm id}$, post-composition with $\Psi_\epsilon$ is sufficient to perturb the inclusion $S_j^+ \hookrightarrow L_{j-1,j}$ and make it transverse to the strata of $\scr V_{i,j}$. Assume without loss of generality that $\{\Psi_t\}$ can be extended smoothly outside of $[0, \epsilon]$, such that $\Psi_t \equiv {\rm id}$ for $t \le 0$, and $\Psi_t \equiv \Psi_\epsilon$ for $t \ge \epsilon$. If $\Psi_t$ is small enough, we claim that it can be incorporated into the flow $\Phi$ by changing the pseudo-gradient $X$ on a small cylindrical slice right below $L_{j-1,j}$, so that the new embedding of $S_j^+ \hookrightarrow L_{j-1,j}$ is equal to the old one, post-composed with $\Psi_\epsilon$.

        To see this, denote $L_{j-1,j}$ by $L^1$, choose another level-set $L^0$ between $L^1$ and $f^{-1}(r_j)$, and let $\{L^t\}_{t \in [0,\epsilon]}$ be the image of $L^0$ under the upward $\Phi$-flow for time $t$. Note that only $L^0$ and $L^1$ are true level-sets (meaning that $f$ is constant on them). Also assume that $\epsilon$ is small enough so that $\{L^t\}_{t \in [0,\epsilon]}$ never reach $L^1$. Further, let us use the notation $\Phi^{t,t'}$ to denote the $\Phi$-flow diffeomorphism needed to get from $L^t$ to $L^{t'}$, for any $t,t' \in [0, \epsilon] \cup \{1\}$. Now, simply define new diffeomorphisms $\{\tilde \Phi^{0,t} : L^0 \to L^t\}_{t \in [0,\epsilon]}$ via the composite $\tilde \Phi^{0,t} := \Phi^{1,t} \circ \Psi_t \circ \Phi^{0,1}$, and define the new pseudo-gradient $\tilde X$ as minus the derivative of $\tilde \Phi^{0,t}$ in the $t$-direction inside the slice between $L^0$ and $L^\epsilon$, and let it be the old $X$ everywhere else. This operation incorporates the diffeomorphism $\Psi_\epsilon$ into the flow, in the desired way. Finally, it is still pseudo-gradient if the original $\Psi_t$ is sufficiently $C^1$-small, finishing the proof of the validity of the first blow-up.
        
        As for the validity of the blow-up ${\rm Bl}^+_{\scr M_{i,j}} \scr V_{i,j}$, which only requires that $\scr M_{i,j} \subset \scr V_{i,j}$ be properly embedded, this in fact follows from {\sc Prop.\,\ref{prop:blfib}} applied to the fiber-product we have just studied. This finishes the proof that the smooth structure on $\scr M_{i,j} \subset \scr V_{i,j}$ is well-defined. It is clear that these moduli spaces agree with the set-theoretic {\sc Defs.\,\ref{def:moduli}}, because at each step the flow lines which do not break at $\scr C_j$ continue on to the next level (these are the points in the interior of the blow-up), whereas the flow lines which break at $\scr C_j$ form the exceptional divisor.

        Finally, for $\scr W_i$ to be well-defined, we need to ensure two things: that the fiber-product defining ${\rm ev}^{-1}U_j$ is cut out transversely, and that the smooth structures on the open cover given by $U_k$'s agree on overlaps. For the first claim, we note that the map ${\rm Bl}^+_{D_j^-} U_j \to L_{j-1,j}$ is already a submersion away from the exceptional divisor (since the flow $\Phi$ is a diffeomorphism,) and that on the exceptional divisor it is the projection onto the $S_j^+$-factor; so transversality is superfluous on the interior, and follows from the transversality condition established two paragraphs ago, at the exceptional divisor. The latter claim follows because, as mentioned in the proof of {\sc Prop.\,\ref{prop:backmap}}, away from the critical values, the inclusion $D_j^- \subset U_j$ is locally the same as $S_j^- \subset L_{j,j+1}$ times an interval. A similar analysis of the exceptional divisors shows that this $\scr W_i$ agrees with {\sc Def.\,\ref{def:moduli}-iii}.
    \end{proof}

    \begin{RMK}\label{rmk:sym}
        The joint product of all the evaluation maps provides a smooth set-theoretic injection from $\scr M_{i,j} \subset \scr V_{i,j}$ into the product $L_{i,i+1} \times \cdots \times L_{j-1,j}$ of all the intermediary level-sets; indeed, a broken flow line is uniquely determined by all these evaluations, due to unique continuation. In fact, this injection turns out to be a smooth embedding. This can be proven by induction on $j$: away from the exceptional divisor, the forgetful map $\scr V_{i,j} \to \scr V_{i,j-1}$ is a diffeomorphism, so the old evaluation maps suffice to obtain an embedding; at the exceptional divisor, the kernel of the differential of $\scr V_{i,j} \to \scr V_{i,j-1}$ consists precisely of the $S_j^-$-directions, which are mapped injectively into the next level-set $L_{j,j+1} \supset S_j^-$. This shows that the smooth structure on $\scr M_{i,j}, \scr V_{i,j}$ is in some sense canonical; it also shows that it is symmetric with respect to reversing the direction of the flow in the case of $\scr M_{i,j}$, even though the blowing-up process is not. In a similar fashion, the $U_j$-portion of $\scr W_i$ is embedded into $L_{i,i+1} \times \cdots \times L_{j-1,j} \times M$ via the evaluation maps, which we leave to the interested reader to check.
    \end{RMK}

    We conclude this section by noting that the moduli spaces $\scr M_{i,j}, \scr W_i$ we have constructed from the data of a generic Morse-Bott flow $(M, f, \Phi)$ naturally give rise to flow categories.

    \begin{EXP}\label{exp:F}
        The various critical sets $\{\scr C_i\}_{1 \le i \le n}$, and moduli spaces $\{\scr M_{i,j}\}_{1 \le i < j \le n}$, form a graded flow category $\scr F_{\rm MB}$ associated to $(M, f, \Phi)$, with Maslov grading $\mu(i) := {\rm rk\;} T_i^-$. Indeed, it only remains to see that \ref{eq:multifib} are transversely cut-out, which in fact follows from the transversality conditions already imposed in {\sc Thm.\,\ref{thm:sm}}, by a simple induction on $i_n$.
    \end{EXP}

    \begin{EXP}\label{exp:Faugm}
        The $\scr F_{\rm MB}$ above can be augmented to a graded flow category $\scr F^{\rm augm}_{\rm MB}$, by adding one extra index $\infty$ formally after $n$, and defining $\scr C_\infty := M$, $\scr M_{i,\infty} := \scr W_i$, with target map $\scr M_{i,\infty} \to M$ given by evaluation at the endpoint and Maslov grading $\mu(\infty) := -1$. This will be useful for both framing $\scr F_{\rm MB}$ in {\sc\S\ref{sec:stfr}}, and proving the isomorphism of {\sc\S\ref{sec:cjs}}.
    \end{EXP}

\section{Stable normal framings for the Morse-Bott flow category}\label{sec:stfr}

    In this section, we produce stable normal framings, in the sense of {\sc Def.\,\ref{def:snf}}, for the flow category $\scr F^{\rm augm}_{\rm MB}$ associated to a generic Morse-Bott flow $(M,f,\Phi)$ of {\sc Ex.\,\ref{exp:Faugm}}. In this sense, we prove the following abstract result {\sc Thm.\,\ref{thm:fr}}; for motivation, $A_{i,j}$ can be taken to be the normal bundles associated to some coherent system of embeddings $\iota_{i,j}$, see {\sc Def.\,\ref{def:cohemb}}, and likewise $N_i := T_i^-$ (other small variations will also be useful, see {\sc Rmk.\,\ref{rmk:other}} below), so the result implies that after stabilizing the embeddings $\iota_{i,j}$ via $X_{i,j}$, and letting $V_{i,j} := Y_{i,j}$, $\xi_i := T_i^-$, we get the desired stable normal framings.

    \begin{THM}\label{thm:fr}
        Let $A_{i,j}$ be vector bundles over $\scr M_{i,j}$, and similarly $A_{i,\infty}$ be vector bundles over $\scr M_{i,\infty} := \scr W_i$ from {\sc Ex.\,\ref{exp:Faugm}}, with associative identifications $A_{i,j} \oplus A_{j,k} \cong A_{i,k}$ covering the inclusions $\scr M_{i,j} \times_j \scr M_{j,k} \subset \scr M_{i,k}$ for all $i<j<k$ in $I \sqcup \{\infty\}$. Let $M_i$ denote the restriction of $A_{i,\infty}$ to the subspace $\scr C_i \subset \scr W_i$, consisting of constant flow lines, and $M_\infty = 0$. Further, let $N_i$ be vector bundles over $\scr C_i$ (by convention, $N_\infty = 0$), and $P_i$ be vector \emph{spaces} (again, $P_\infty = 0$) with identifications $M_i \oplus N_i = P_i$. Then, there exist vector spaces $X_{i,j}$ and $Y_{i,j}$ with associative identifications $X_{i,j} \oplus X_{j,k} = X_{i,k}$ and $Y_{i,j} \oplus Y_{j,k} = Y_{i,k}$ for all $i < j < k$ in $I \sqcup \{\infty\}$ whose definitions depend only on the $P_i$, together with vector bundle isomorphisms
        \[ \phi_{i,j} : N_i \oplus A_{i,j} \oplus X_{i,j} \overset\sim\to Y_{i,j} \oplus N_j, \]
        over $\scr M_{i,j}$ satisfying the following commutativity condition for all $i < j < k$ in $I \sqcup \{\infty\}$:
        \begin{equation*}\numeq\label{eq:hex} \begin{tikzcd}
            N_i \oplus A_{i,j} \oplus X_{i,j} \oplus A_{j,k} \oplus X_{j,k} \rar["\phi_{i,j} \oplus {\rm id}"]\dar["{\rm swap}"] & Y_{i,j} \oplus N_j \oplus A_{j,k} \oplus X_{j,k} \rar["{\rm id} \oplus \phi_{j,k}"] & Y_{i,j} \oplus Y_{j,k} \oplus N_k \dar["{\rm mult}"] \\
            N_i \oplus A_{i,j} \oplus A_{j,k} \oplus X_{i,j} \oplus X_{j,k} \rar["{\rm mult}"] & N_i \oplus A_{i,k} \oplus X_{i,k} \rar["\phi_{i,k}"] & Y_{i,k} \oplus N_k,
        \end{tikzcd} \end{equation*}
        over $\partial^j \scr M_{i,k}$. Moreover, this can be achieved up to a contractible space of choices.
    \end{THM}

    To prove the result, we need a few preliminaries:

    \begin{LEM}\label{lem:trivw}
        The space which parameterizes all collections $\{\psi_i : A_{i,\infty} \overset\sim\to M_i\}_{i \in I}$ of bundle isomorphisms covering the source maps $\scr W_i \to \scr C_i$, satisfying the following two axioms:
        \begin{itemize}
            \item[\sc i.] The restriction $\psi_i|_{\scr C_i}$ to the space of constant flow lines is the identity;
            \item[\sc ii.] Given pairs $(x, y), (x, z) \in \scr M_{i,j} \times_j \scr W_j = \partial^j \scr W_i$, sharing the same first entry $x$, there is an equality of transition functions
            \[ \psi_i|_{(x,z)}^{-1} \circ \psi_i|_{(x,y)} = {\rm id}_{A_{i,j}|_x} \oplus (\psi_j|_z^{-1} \circ \psi_j|_y) \]
            from $A_{i,\infty}|_{(x, y)} = A_{i,j}|_x \oplus A_{j,\infty}|_y$ to $A_{i,\infty}|_{(x, z)} = A_{i,j}|_x \oplus A_{j,\infty}|_z$;
        \end{itemize}
        is contractible.
    \end{LEM}

    \begin{proof}
        We proceed by descending induction on $i \in I$; given a specific $i \in I$, let us assume that isomorphisms $\psi_j$ for $j > i$ have all been constructed, satisfying conditions {\sc i} and {\sc ii} above, and let us show that the space of $\psi_i$ satisfying these same two conditions is also contractible.

        Consider the topological quotient $\widetilde{\scr W}_i$ of $\scr W_i$ which identifies any two pairs $(x, y) \sim (x, z)$ on any boundary component $\partial^j \scr W_i$ which share the same first entry. Condition {\sc ii} implies that the transition functions ${\rm id}_{A_{i,j}|_x} \oplus (\psi_j|_{z}^{-1} \circ \psi_j|_{y})$ for all $j$, and all $x, y, z$, define a descent datum for the bundle $A_{i,\infty}$ along the quotient map $\scr W_i \twoheadrightarrow \widetilde{\scr W}_i$. Let us denote the descended bundle on $\widetilde{\scr W}_i$ by $\widetilde A_{i, \infty}$; it is locally trivial at any orbit $(x, *) \subset \partial^j \scr W_i$, since one can use $\psi_j$ to trivialize $A_{i,\infty}|_{(x,*)}$ in a neighborhood thereof inside $\partial^j \scr W_i$, and then use the fact that $\partial^j \scr W_i \subset \scr W_i$ is a cofibration to extend the trivialization all the way to a neighborhood of $(x,*) \subset \scr W_i$. Note that an isomorphism $\psi_i : A_{i, \infty} \overset\sim\to M_i$ satisfies condition {\sc ii} if and only if it factors through the descended bundle $\widetilde A_{i, \infty}$. So the claim we wish to prove is equivalent to showing that the space of bundle isomorphisms $\tilde \psi_i : \widetilde A_{i,\infty} \overset\sim\to M_i$ covering the source map to $\scr C_i$, which satisfy condition {\sc i}, is contractible.

        To this end, we observe that $\widetilde{\scr W}_i$ is homeomorphic to the disk bundle associated to the vector bundle $T_i^-$ over $\scr C_i$, and in particular that the inclusion $\scr C_i \subset \widetilde{\scr W}_i$ of the subspace of constant paths is a homotopy equivalence. Indeed, if we define
        \[ \hat D_i^- := S_i^- \times [-\infty, \infty] \smash{\;\big/\,} (x,-\infty) \sim (y, -\infty),\, \forall x, y \text{ in the same fiber}, \]
        then there is a map $\scr W_i \to \hat D_i^-$ which takes a flow line $\gamma$ from $\scr C_i$ terminating at some point $x$ to the intersection of $\gamma$ with $S_i^-$ in the first entry, (prolong $\gamma$ if it is too short, so that it intersects $S_i^-$) and to the length of time needed to flow from $L_{i,i+1}$ to $x$ in the second entry (with a negative sign if one has to flow backwards, and of infinite magnitude if it cannot be reached in finite time.) This is continuous by standard ODE theory, and descends to the quotient $\widetilde {\scr W}_i \to \hat D_i^-$ because broken flow lines always take an infinite amount of time to travel. In fact, this is a continuous bijection between compact Hausdorff spaces, hence a homeomorphism, as claimed. The desired conclusion now follows from the fact that sections of any bundle can be extended along an acyclic cofibration up to a contractible space of choices; so in particular we can extend $\tilde\psi_i$ from $\scr C_i$ to the whole $\widetilde{\scr W}_i$.
    \end{proof}

    \begin{COR}\label{cor:relstfr}
        There is a contractible space parameterizing associative bundle isomorphisms $\alpha_{i,j} : A_{i,j} \oplus M_j \overset\sim\to M_i$ over $\scr C_i$ for all $i, j \in I$. These can be further extended associatively to $\alpha_{i,\infty} : A_{i,\infty} \oplus M_\infty \overset\sim\to M_i$, where $M_\infty := 0$ by convention.
    \end{COR}

    \begin{proof}
        In virtue of {\sc Lem.\,\ref{lem:trivw}}, we can replace $A_{i,\infty}$ by $M_i$ and $A_{j,\infty}$ by $M_j$ in the splitting $A_{i,j} \oplus A_{j,\infty} \cong A_{i,\infty}$, to get $\alpha_{i,j} := \psi_i \circ ({\rm id}_{A_{i,j}} \oplus \psi_j^{-1})$. To ensure that the definition of $\alpha_{i,j}$ does not depend on the particular point in $\scr W_j$ at which we consider the isomorphism, let us investigate the transition function at two points $y$ and $z$ in $\scr W_j$:
        \[ \alpha_{i,j}|_z^{-1} \circ \alpha_{i,j}|_y = ({\rm id}_{A_{i,j}} \oplus \psi_j|_z) \circ \psi_i|_z^{-1} \circ \psi_i|_y \circ ({\rm id}_{A_{i,j}} \oplus \psi_j|_y^{-1}). \]
        By property {\sc ii}, the map above can be further rewritten as
        \[ {\rm id}_{A_{i,j}} \oplus (\psi_j|_z \circ \psi_j|_z^{-1} \circ \psi_j|_y \circ \psi_j|_y^{-1}) = {\rm id}|_{A_{i,j} \oplus M_j}, \]
        as desired. The associativity condition also follows from a simple computation:
        \[ \begin{aligned}
            \alpha_{i,j} \circ ({\rm id}_{A_{i,j}} \oplus \alpha_{j,k}) &= \psi_i \circ ({\rm id}_{A_{i,j}} \oplus \psi_j^{-1}) \circ \left[{\rm id}_{A_{i,j}} \oplus \big(\psi_j \circ ({\rm id}_{A_{j,k}} \oplus \psi_k^{-1})\big)\right] \\
            &= \psi_i \circ ({\rm id}_{A_{i,k}} \oplus \psi_k^{-1}) = \alpha_{i,k}.
        \end{aligned} \]
        Finally, we define $\alpha_{i,\infty} : A_{i,\infty} \oplus M_\infty \overset\sim\to M_i$ as simply $\psi_i$ itself. The associativity condition for $k = \infty$ simply follows from the definition of $\alpha_{i,j}$.
    \end{proof}

    \begin{proof}[Proof of {\sc Thm.\,\ref{thm:fr}}]
        Throughout the proof, the indices $i,j,k$ run through $I \sqcup \{\infty\}$. Multiplying $\alpha_{i,j} : A_{i,j} \oplus M_{j} \overset\sim\to M_{i}$ of {\sc Cor.\,\ref{cor:relstfr}} by $(-1)$, and adding $N_i \oplus N_j$ to both sides yields an isomorphism
        \[ \beta_{i,j} : N_i \oplus A_{i,j} \oplus P_j \overset\sim\to P_i \oplus N_j. \]
        Let us set $X_{i,j} := P_{i+1} \oplus \cdots \oplus P_j$ and $Y_{i,j} := P_i \oplus \cdots \oplus P_{j-1}$ with the obvious associative identifications $X_{i,j} \oplus X_{j,k} = X_{i,k}$ and $Y_{i,j} \oplus Y_{j,k} = Y_{i,k}$. Further adding $P_{i+1} \oplus \cdots \oplus P_{j-1}$ to both sides of the map $\beta_{i,j}$ above gives a map $\phi_{i,j}$ with the desired domain and codomain
        \[ \phi_{i,j} : N_i \oplus A_{i,j} \oplus X_{i,j} = N_i \oplus A_{i,j} \oplus P_{i+1} \oplus \cdots \oplus P_{j} \overset\sim\to P_i \oplus \cdots \oplus P_{j-1} \oplus N_j \cong Y_{i,j} \oplus N_j. \]
        However, these $\phi_{i,j}$ do not satisfy the desired commutativity condition \ref{eq:hex} strictly, but rather only up to homotopy. We briefly explain why, and how this can be remedied. The two composites in the square agree strictly on the summands $P_{i+1}, \ldots, P_{j-1}$, as well as $P_{j+1}, \ldots, P_{k-1}$, which were added on artificially; so it suffices to analyze the two composites on the remaining summands:
        \[ \begin{tikzcd}
            N_i \oplus A_{i,j} \oplus P_j \oplus A_{j,k} \oplus P_k \rar["\beta_{i,j} \oplus {\rm id}"]\dar["{\rm mult}"] & P_i \oplus N_j \oplus A_{j,k} \oplus P_k \dar["{\rm id} \oplus \beta_{j,k}"] \\
            N_i \oplus A_{i,k} \oplus P_j \oplus P_k \rar["\beta_{i,k} \oplus {\rm id}"] & P_i \oplus P_j \oplus N_k
        \end{tikzcd} \]
        Further unpacking $P_i = M_i \oplus N_i$, $P_j = M_j \oplus N_j$, and $P_k = M_k \oplus N_k$, the two composites in question agree on the $N$-summands, so we can further simplify the question to checking the commutativity of
        \[ \begin{tikzcd}
            A_{i,j} \oplus M_j \oplus A_{j,k} \oplus M_k \rar["-\alpha_{i,j}\oplus{\rm id}"]\dar["{\rm mult}"] & M_i \oplus A_{j,k} \oplus M_k \dar["{\rm id}\oplus-\alpha_{j,k}"] \\
            A_{i,k} \oplus M_k \oplus M_j \rar["-\alpha_{i,k} \oplus {\rm id}"] & M_i \oplus M_j.
        \end{tikzcd} \]
        There is an obvious homotopy $F^{i,j,k}_{t} = F_t$
        \[ F_t : A_{i,j} \oplus M_j \oplus A_{j,k} \oplus M_k \overset{{\rm id} \oplus -\alpha_{j,k}}{-\!\!\!-\!\!\!-\!\!\!\longrightarrow} A_{i,j} \oplus M_j \oplus M_j \overset{{\rm id} \oplus {\rm Rot}_{{t\pi}/2} \otimes M_j}{-\!\!\!-\!\!\!-\!\!\!-\!\!\!-\!\!\!\longrightarrow} A_{i,j} \oplus M_j \oplus M_j \overset{-\alpha_{i,j} \oplus {\rm id}}{-\!\!\!-\!\!\!-\!\!\!\longrightarrow} M_i \oplus M_j \]
        where ${\rm Rot}_\theta \in {\rm GL}(\bb R^2)$ is the rotation matrix by angle $\theta$. Indeed, $F_0$ recovers the right-down composite in the square, because the two $\alpha$-maps work in parallel, whereas $F_1$ recovers the down-right composite, due to the associativity of the $\alpha$-maps. The way to fix the maps $\phi_{i,j}$ in order that the commutativity \ref{eq:hex} might be strictly satisfied is to ``push in'' the maps $\phi_{i,j}$ in the interior of the moduli spaces $\scr M_{i,j}$, and make extra room in order to incorporate the homotopy $F_t$ above into the map itself. The choices involved in this process are:
        \begin{itemize}
            \item[\sc i.] a collar neighborhood $\partial^j\scr M_{i,k} \times [0, 2] \to \scr M_{i,k}$ prescribing a coordinate $t \in [0, 2]$ near the $j$-boundary;
            \item[\sc ii.] a trivialization of $A_{i,k}$ on the collar neighborhood, relative to the boundary inclusion $\{0\} \subset [0,2]$, i.e. isomorphisms $A_{i,k}|_{(x, t)} \overset\sim\to A_{i,k}|_{(x,0)}$.
        \end{itemize}
        One then defines a push-off denoted $\scr P^j\phi_{i,k}$ which agrees with $\phi_{i,k}$ away from the collar neighborhood, and takes values
        \[ \scr P^j\phi_{i,k}(x,t) := \begin{cases}
            F^{i,j,k}_t(x) & \text{if } t \in [0,1] \\
            \phi_{i,k}(x, 2t-2) & \text{if } t \in [1,2]
        \end{cases} \]
        on the collar neighborhood. This push-off procedure $\scr P^j$ makes the $\phi_{i,j}, \phi_{j,k}, \phi_{i,k}|_{\partial^j}$ satisfy the compatibility strictly, rather than up to homotopy. Finally, we must explain how to fix \emph{all} the $\phi$'s simultaneously, so that all the constraints are strictly satisfied. For this, one must impose suitable compatibility conditions on the choices {\sc i} and {\sc ii} between various boundaries $\partial^j$; namely that the collar neighborhoods of {\sc i} restrict unambiguously to multi-collar neighborhoods $\partial^J \scr M_{i,k} \times [0,2]^J \to \scr M_{i,k}$ for any subset $J$ of indices, and that the trivializations of {\sc ii} commute, i.e. are independent in the order in which we perform them, to get from an interior point to the $\partial^J$-stratum. Then, since the rotations in the definition of $F_t$ are performed independently for different $j, j'$, one can redefine $\phi_{i,k}$ by applying all the push-off operations $\scr P^j$ for $i < j < k$ unambiguously to $\phi_{i,k}$. We note that the space of choices of data {\sc i} and {\sc ii} satisfying the desired compatibility conditions is contractible, which is proved by standard techniques.
    \end{proof}

    \begin{COR}\label{cor:fr}
        Let $\iota_{i,j} : \scr M_{i,j} \hookrightarrow \scr C_i \times \bb E_{i,j}$ be a coherent system of embeddings, see {\sc Def.\,\ref{def:cohemb}}, of the Morse-Bott flow category $\scr F^{\rm augm}_{\rm MB}$ of {\sc\S\ref{sec:smooth}}. After stabilizing the embeddings by certain vector spaces $X_{i,j}$, there is a contractible parameter-space of stable normal framings of $\scr F^{\rm augm}_{\rm MB}$ relative to $\xi_i := T_i^-$ and $\xi_\infty := 0$, see {\sc Def.\,\ref{def:snf}}.
    \end{COR}

    \begin{proof}
        Let $A_{i,j}$ be the normal bundles to the embeddings $\iota_{i,j}$, with associative identifications $A_{i,j} \oplus A_{j,k} = A_{i,k}$ as explained in {\sc Rmk.\,\ref{rmk:normadd}}. Further, let $N_i := T_i^-$, $P_i := T\bb E_{i,\infty}$. To see why $M_i \oplus N_i \cong P_i$, note that since the embedding $\scr W_i = \scr M_{i,\infty} \hookrightarrow \scr C_i \times \bb E_{i,\infty}$ respects the source map $\scr W_i \to \scr C_i$, it follows that, when restricted to the neighborhood $D_i^- \subset \scr W_i$ of the subspace $\scr C_i \subset \scr W_i$ of constant flow-lines, the embedding $\iota_{i,j}$ maps the fibers $D_i^-|_p$ inside $\{p\} \times \bb E_{i,\infty}$ for all $p \in \scr C_i$; so the normal bundle $A_{i,\infty}$ is equal to the normal bundle of $D_i^-|_p \subset \{p\} \times \bb E_{i,\infty}$, i.e. it fits into a short-exact sequence
        \[ 0 \to T(D_i^-|_p) \to T\bb E_{i,\infty} \to A_{i,\infty}|_p \to 0. \]
        Splitting this short-exact sequence gives the desired $M_i \oplus N_i = P_i$. Now, applying {\sc Thm.\,\ref{thm:fr}} to these data produces the stable normal framings
        \[ \phi_{i,j} : \xi_i \oplus (A_{i,j} \oplus X_{i,j}) \overset\sim\to V_{i,j} \oplus \xi_j \]
        of the stabilization of the embeddings $\iota_{i,j}$ by $X_{i,j}$, relative $\xi_i = T_i^-, \xi_\infty = 0$ and $V_{i,j} := Y_{i,j}$.
    \end{proof}

    \begin{RMK}\label{rmk:other}
        One can slightly modify the definitions of $A_{i,j}, N_i, P_i$ of {\sc Thm.\,\ref{thm:fr}}, to obtain other important results. For instance, if $\scr F_{\rm MB}$ is classical, i.e. all $\scr C_i$ are points, then the bundles $M_i$ are already trivial, and so there is no need to add $N_i$ to trivialize them; instead, one can set $N_i = 0$, and obtain a simpler notion of stable framing, where all the $\xi_i$ are zero. Another important application follows if we replace $A_{i,\infty}$ with $A_{i,\infty} \oplus {\rm ev}^* F$, for some vector bundle $F$ on $M$. Then, {\sc Thm.\,\ref{thm:fr}} still gives a stable normal framing of $\scr F_{\rm MB}$, although it does not extend over $\scr F_{\rm MB}^{\rm augm}$ immediately, because the isomorphism $\phi_{i,\infty} : \xi_i \oplus A_{i,\infty} \oplus X_{i,\infty} \oplus F \overset\sim\to Y_{i,\infty} \oplus 0$ does not fit the definition of stable normal framing, as ${\rm ev}^*F$ is not necessarily trivial. To this end, let us choose a complement bundle $E$ on $M$, with a trivialization of $E \oplus F$. Then, adding $E$ to both sides, one obtains an isomorphism
        \[ \phi_{i,\infty} : \xi_i \oplus (A_{i,\infty} \oplus X_{i,\infty}) \oplus (E \oplus F) \overset\sim\to Y_{i,\infty} \oplus E. \]
        So, following the proof of {\sc Cor.\,\ref{cor:fr}}, after stabilizing the embeddings $\iota_{i,j}$ by $X_{i,j}$ when $j < \infty$, and by $X_{i,\infty} \oplus (E \oplus F)$ when $j = \infty$, we obtain stable normal framings of $\scr F_{\rm MB}^{\rm augm}$ relative $\{\xi_i\}_{i \in I}$ from before, and now $\xi_\infty = E$ instead of $\xi_\infty = 0$.
    \end{RMK}

    \begin{DEF}\label{def:can}
        We call a stable normal framing of $\scr F_{\rm MB}$ (relative arbitrary bundles $\xi_i$) \keywd{$E$-twisted}, if it extends over $\scr F_{\rm MB}^{\rm augm}$ with $\xi_\infty = E$. {\sc Cor.\,\ref{cor:fr}} and its extension {\sc Rmk.\,\ref{rmk:other}}, guarantee the existence of such $E$-twisted framings for any $E$, and even offer a contractible parameter-space of such choices, which we refer to as the \emph{canonical} $E$-twisted framings. Since stabilizing all the $\xi_i$ by a common vector space will not change the stable CJS construction, see {\sc Def.\,\ref{def:cjs}} ahead, we can talk about $E$-twisted framings for a reduced $K$-theory class $[E] \in \widetilde{KO}(M)$, as we did in the Introduction.
    \end{DEF}

    We end this section by showing that \emph{every} stable normal framing of $\scr F_{\rm MB}$ is $E$-twisted for some vector bundle $E$, after suitable stabilization. The first step is to figure out what the right bundle $E$ is, which we accomplish in the following definition:
    \begin{DEF}\label{def:asb}
        Consider a stable normal framing of $\scr F_{\rm MB}$, following the same notations as in {\sc Def.\,\ref{def:snf}}. After potentially stabilizing the coherent system of embeddings, one can extend it to $\scr F_{\rm MB}^{\rm augm}$; we also choose normal bundles $N_{i,j}$ for all $i < j$ in $I \sqcup \{\infty\}$, as in {\sc Rmk.\,\ref{rmk:normadd}}. Then, there is a vector bundle on the disjoint union $\coprod_{i \in I} \scr W_i$, given by $V_{1,i} \oplus \xi_i \oplus N_{i,\infty}$ on the $i^{\rm th}$ summand. We can descend this bundle along the disjoint union of the evaluation maps $\coprod_{i \in I} \scr W_i \overset{\rm ev}\to M$ at the free endpoints, by defining the following quotient:
        \[ E := \coprod_{i \in I} {V_{1,i} \oplus \xi_i \oplus N_{i,\infty}} \,\Big/\!\sim, \]
        where the equivalence relation identifies fibers in the $\partial^j$-boundary of the $i^{\rm th}$ summand to fibers in the $j^{\rm th}$ summand, via the stable normal framings
        \[ V_{1,i} \oplus \xi_i \oplus N_{i,\infty} = V_{1,i} \oplus \xi_i \oplus N_{i,j} \oplus N_{j,\infty} \overset{\phi_{i,j}}\longrightarrow V_{1,i} \oplus V_{i,j} \oplus \xi_j \oplus N_{j,\infty} = V_{1,j} \oplus \xi_j \oplus N_{j,\infty}. \]
        The equivalence relation is transitive precisely because of the boundary compatibility condition \ref{eq:snf} imposed on $\phi_{i,j}$. This defines a vector bundle on the original manifold $M$, which we call the \emph{associated bundle} to the stable normal framing.
    \end{DEF}
    \begin{THM}\label{thm:ofr}
        Any stable normal framing of the flow category $\scr F_{\rm MB}$ constructed in {\sc\S\ref{sec:smooth}} for a generic Morse-Bott flow becomes $E$-twisted after suitable stabilization, where $E$ is the associated bundle to the framing, introduced in {\sc Def.\,\ref{def:asb}}.
    \end{THM}
    \begin{proof}
        Let us stabilize the bundles $\xi_i$ by adding $V_{1,n}$ to all of them: $\xi_i' := V_{1,n} \oplus \xi_i$. For $i < j < \infty$, we define the obvious $\phi'_{i,j} := {\rm id}_{V_{1,n}} \oplus \phi_{i,j}$. We will now extend this stable normal framing to $\scr F_{\rm MB}^{\rm augm}$, with $\xi_\infty = E$ the associated bundle, and $V_{i,\infty} = V_{i,n}$ (or equivalently $V_{n,\infty} = 0$). So we must construct maps
        \[ \phi_{i,\infty} : V_{1,n} \oplus \xi_i \oplus N_{i,\infty} \overset\sim\to V_{i,n} \oplus E, \]
        satisfying the boundary compatibility condition coming from \ref{eq:snf}. There is a quite obvious choice of such $\phi_{i,\infty}$, inspired by the definition of $E$. One can decompose the domain as $V_{i,n} \oplus V_{i,1} \oplus \xi_i \oplus N_{i,\infty}$, and use the fact that $E$ pulled back to $\scr W_i$ is identified by definition with the last three terms $V_{1,i} \oplus \xi_i \oplus N_{i,\infty}$. For reasons very similar to those in the proof of {\sc Thm.\,\ref{thm:fr}}, we will introduce a sign on the $V_{i,n}$ summand:
        \[ \phi_{i,\infty} := (-{\rm id}_{V_{i,n}}) \oplus {\rm id}_E : V_{i,n} \oplus (V_{1,i} \oplus \xi_i \oplus N_{i,\infty}) \to V_{i,n} \oplus E. \]
        Also similarly to {\sc Thm.\,\ref{thm:fr}}, the compatibility relation \ref{eq:snf} will only be satisfied up to homotopy, which will have to be fixed by the same pushing-in technique. Concretely, we would like to show that $\phi_{i,\infty}|_{\partial^j}$ is given by the composite
        \[ (V_{j,n} \oplus V_{i,j} \oplus V_{1,i}) \oplus \xi_i \oplus N_{i,j} \oplus N_{j,\infty} \overset{\phi_{i,j}}\to (V_{j,n} \oplus V_{i,j} \oplus V_{1,i}) \oplus V_{i,j} \oplus \xi_j \oplus N_{j,\infty} \overset{\phi_{j,\infty}}\to V_{i,n} \oplus E. \]
        It is not difficult to see that $\phi_{i,\infty}$ is actually given by the composite where we stick in the automorphism $\left[\begin{smallmatrix}0&1\\-1&0\end{smallmatrix}\right] \otimes {\rm id}_{V_{i,j}}$ on the summand $V_{i,j}^{\oplus 2}$ in the middle term. Therefore, one can interpolate between them by using ${\rm Rot}(-t\pi/2) \otimes {\rm id}_{V_{i,j}}$ instead, to witness the homotopy. As in the proof of {\sc Thm.\,\ref{thm:fr}}, on the higher boundary strata $\partial^{u_1, \ldots, u_k, j} \scr W_i$, one can do the rotation independently on each of the summands $V_{i, u_1}, V_{u_1, u_2}, \ldots, V_{u_k,j}$, and change the definition of $\phi_{i,\infty}$ by a very similar kind of push-off.
    \end{proof}

\section{The CJS construction, and recovering Thom spectra on manifolds}\label{sec:cjs}

    In this section we recall the Cohen-Jones-Segal construction, which produces a stable homotopy type from the data of a flow category, see {\sc Def.\,\ref{def:flow}}, together with a coherent system of embeddings $\iota_{i,j}$, see {\sc Def.\,\ref{def:cohemb}} and crucially a stable normal framing thereof, see {\sc Def.\,\ref{def:snf}}. When applied to $\scr F_{\rm MB}$, the flow category associated to a generic Morse-Bott flow $(M, f, \Phi)$ of {\sc Def.\,\ref{exp:F}}, and any set of $E$-twisted stable normal framings, see {\sc Def.\,\ref{def:can}}, for some reduced $K$-theory class $[E] \in \widetilde{KO}(M)$, the CJS construction recovers the stable homotopy type of the Thom spectrum $M^E$, defined as the formal desuspension $\Sigma^{-{\rm rk\;}E} {\rm Th}(E)$ for any actual vector bundle $E$ representing the class $[E]$.
    
    In this section, we restrict ourselves to using the Spanier-Whitehead model which employs formal (de)suspensions $\Sigma^{-n} X$ of pointed spaces ($n \in \bb Z_{>0}$) to represent finite stable homotopy types \cite[\sc\S XII.3]{Whi},\cite{SW}. Unpointed spaces $X$ can be made pointed by one-point compactification $X^+$.

    \begin{DEF}\label{def:ptmaps}
        Given a coherent collection of embeddings $\iota_{i,j}$ of a flow category $\scr F$, following the same notations as in {\sc Defs.\,\ref{def:flow},\,\ref{def:cohemb}}, we define a coherent collection of tubular neighborhoods to be an extension of the $\iota_{i,j}$ to diffeomorphisms
        \[ \iota_{i,j} : N_{i,j} \overset\sim\to U_{i,j} \subset \scr C_i \times \bb E_{i,j}, \]
        where $N_{i,j}$ are normal bundles to the $\scr M_{i,j}$, see {\sc Rmk.\,\ref{rmk:normadd}}, and $U_{i,j}$ are open subsets of the ambient spaces, satisfying the same coherence condition as in {\sc Def.\,\ref{def:cohemb}}.
        In conjunction with stable normal framings $\phi_{i,j} : \xi_i \oplus N_{i,j} \overset\sim\to V_{i,j} \oplus \xi_j$ of the original embeddings, these induce a coherent system of \keywd{Pontrjagyn-Thom collapse maps} ${\rm PT}_{i,j} : {\rm Th}(\xi_i) \wedge \bb E_{i,j}^+ \to V_{i,j}^+ \wedge {\rm Th}(\xi_j)$, defined by mapping via $\phi_{i,j}$ inside the subset $\xi_i \times_{\scr C_i} U_{i,j}$, and mapping to the basepoint everywhere else. These PT-maps satisfy the compatibility condition that the boundary restriction ${\rm PT}_{i,j}|_{\partial^u} : {\rm Th}(\xi_i) \wedge \bb E_{i,u}^+ \wedge \bb E_{u,j}^+ \to V_{i,u}^+ \wedge V_{u,j}^+ \wedge {\rm Th}(\xi_j)$ is equal to 
        \[\numeq\label{eq:ptassoc} {\rm PT}_{i,j}|_{\partial^u} = ({\rm id}_{V_{i,u}^+} \wedge {\rm PT}_{u,j}) \circ ({\rm PT}_{i,u} \wedge {\rm id}_{\bb E_{u,j}^+}). \]
    \end{DEF}

    \begin{PROP}\label{prop:stabcontr}
        As the minimum dimension $N$ of all the $\bb E_{i,j}$ goes to $\infty$, the space parameterizing all the coherent choices of embeddings and tubular neighborhoods of any given flow category becomes non-empty, and in fact its connectivity also goes to $\infty$.
    \end{PROP}

    \begin{proof}
        The result follows from standard differential topology techniques, particularly the Whitney embedding theorem relative to boundary, and the uniqueness of tubular neighborhoods. We leave the details to the interested reader; \cite{ManSar} also briefly discusses this.
    \end{proof}

    \begin{NOT}
        Define the Euclidean $\{i < u \le j\}$-space $\bb F_{i, j} := \bb E_{i,j} \times \bb R_{\ge 0}^{\{j\}}$ when $i < j$, and $\bb F_{i,i} := \{0\}$. We also let $V_{i,i} = \{0\}$ by convention.
    \end{NOT}

    \begin{DEF}\label{def:cjs}
        The \keywd{unstable Cohen-Jones-Segal construction} associated to the data of {\sc Def.\,\ref{def:ptmaps}} is defined to be the quotient
        \begin{equation*}\numeq\label{eq:cjs}
            {\rm CJS}^u(\scr F, {\rm PT}_{i,j}) := \left(\coprod_{i \in I} V_{\min I,i}^+ \wedge {\rm Th}(\xi_i) \wedge \bb F_{i, \max I}^+\right) \Big/ \sim,
        \end{equation*}
        where we quotient by the equivalence relation which identifies points in the $\partial^j$-boundary of the $i^{\rm th}$ summand to points in the $j^{\rm th}$ summand via the map
        \begin{equation*}\numeq\label{eq:eqrel}
            \begin{gathered}
                V_{\min I,i}^+ \wedge {\rm Th}(\xi_i) \wedge \partial^j\bb F_{i, \max I}^+ \cong V_{\min I,i}^+ \wedge {\rm Th}(\xi_i) \wedge \bb E_{i,j}^+ \wedge \bb F_{j, \max I}^+ \\
                \overset{{\rm id} \wedge {\rm PT}_{i,j} \wedge {\rm id}}\longrightarrow V_{\min I,i}^+ \wedge V_{i,j}^+ \wedge {\rm Th}(\xi_j) \wedge \bb F_{j, \max I}^+ \cong V_{\min I,j}^+ \wedge {\rm Th}(\xi_j) \wedge \bb F_{j, \max I}^+.
            \end{gathered}
        \end{equation*}
        These identifications are transitive on the pairwise intersections $\partial^j \cap \partial^k$, in virtue of \ref{eq:ptassoc}, so we truly get an equivalence relation without having to take transitive closure. We will write ${\rm CJS}^u(\scr F)$ instead of ${\rm CJS}^u(\scr F, {\rm PT}_{i,j})$ for brevity. If $\scr F$ is graded, we similarly define the \keywd{stable Cohen-Jones-Segal construction} ${\rm CJS}(\scr F)$ to be the stable homotopy type obtained by formally desuspending \ref{eq:cjs} by 
        \[\numeq\label{eq:desusp} \dim V_{\min I,\max I}^+ - \mu(\max I) + {\rm rk}\; \xi_{\max I}, \]
        which is an integer (as opposed to a locally constant function on $\scr C_{\max I}$, see the condition imposed in {\sc Def.\,\ref{def:snf}}).

        To understand the relevance of the first term $\dim V_{\min I,\max I}^+$, consider what happens if we stabilize the embeddings $\iota_{i,j}$ by adding a vector space $W$ to all the $\bb E_{i,j}$ and $V_{i,j}$ with $i < u_0 \le j$ for some fixed index $u_0$, extending the respective tubular neighborhoods $U_{i,j}$ by taking Cartesian product with $W$, and similarly extending the framings $\phi_{i,j}$ by the identity map on the $W$-factor: the maps \ref{eq:eqrel} also change as a result of the stabilization, by being smashed with the identity map $W^+ \to W^+$, so that the unstable CJS construction is also smashed with $W^+$. In particular, the \emph{stable} CJS construction remains unchanged (up to isomorphism in the Spanier-Whitehead category) if we stabilize the embeddings. Similarly, stabilizing all the $\xi_i$ by a common vector space $W$ results in an overall suspension of the unstable CJS construction by $W^+$, which explains the term ${\rm rk}\; \xi_{\max I}$. The last term $-\mu(\max I)$ is there to make the homology of the stable CJS construction agree with the Morse homology, see {\sc Prop.\,\ref{prop:hom}} below.
    \end{DEF}

    \begin{EXP}
        If $\scr F$ is a flow category with only one critical locus $\scr C$, stable normal framing relative $\xi$, and grading $\mu$, then the stable CJS construction is isomorphic to $\Sigma^{\mu - {\rm rk\;} \xi} {\rm Th}(\xi)$.
    \end{EXP}

    \begin{EXP}
        If, instead, $\scr F$ is a flow category with only two critical \emph{points} $p < q$, stably framed manifold $\scr M_{p,q}$ of dimension $n = \mu(p) - \mu(q) - 1$, then the stable CJS construction is isomorphic to the cone on $[\scr M_{p,q}] \in \Omega^{\rm fr}_n$, viewed as a stable map $S^{\mu(p) - 1} \to S^{\mu(q)}$.
    \end{EXP}

    \begin{PROP}[Homotopy invariance]\label{prop:hinv}
        The stable Cohen-Jones-Segal construction is stable homotopy-invariant under homotopies of the data of {\sc Def.\,\ref{def:ptmaps}}.
    \end{PROP}

    \begin{proof}
        There are two key observations from which this result easily follows. One is that the CJS construction can be done parametrically; that is, if there is a compact parameter space $K$ parameterizing embeddings and tubular neighborhoods $\iota_{i,j}$, and framings $\phi_{i,j}$, then we may simply add a factor of $K^+$ in every summand, and have the ${\rm PT}_{i,j}$ depend on $K$:
        \[\left(\coprod_{i \in I} V_{\min I,i}^+ \wedge {\rm Th}(\xi_i) \wedge \bb F_{i, \max I}^+ \wedge K^+\right) \Big/ \sim,\]
        Denote this resulting space by $X$.
        
        The second observation is that $X$ admits a natural filtration
        \begin{equation*}\numeq\label{eq:cjsfilt}
            X = X_{\min I} \supset \cdots \supset X_i \supset \cdots X_{\max I} \supset X_{\max I + 1} := 0
        \end{equation*}
        where $X_j$ is the union of only those summands corresponding to $i \ge j$. The filtration has associated graded $X_i / X_{i+1} \cong V_{\min I,i}^+ \wedge {\rm Th}(\xi_i) \wedge (\bb F_{i, \max I}/\bigcup_j\partial^j)^+ \wedge K^+$. Also, this filtration respects inclusions $K' \subset K$ of parameter spaces. In particular, if $K' \hookrightarrow K$ is a homotopy equivalence, the induced maps on the level of graded quotients $X_i/X_{i+1}$ are also homotopy-equivalences. From here, the standard inductive 5-lemma argument applied to the long-exact sequences on $\pi^{\rm st}_*$ induced by the cofiber sequences $X_i \to X_{i+1} \to X_{i+1}/X_i$ shows that the two parametric CJS constructions are homotopy equivalent. The result now follows by taking $K = [0, 1]$ and $K' = \{0\}$ or $K' = \{1\}$.
    \end{proof}

    In virtue of {\sc Prop.\,\ref{prop:stabcontr}} and {\sc Cor.\,\ref{cor:fr}} (or its generalizations in {\sc Rmk.\,\ref{rmk:other}}), this shows that the stable CJS construction associated to the flow category $\scr F_{\rm MB}$ of {\sc Ex.\,\ref{exp:F}}, using the canonical $E$-twisted framings, see {\sc Def.\,\ref{def:can}}, is independent of the extra choices. Before proving our main result, we show that for any flow category $\scr F$ with extra data as in {\sc Def.\,\ref{def:ptmaps}}, the homology $H_*({\rm CJS}(\scr F); \bb Z)$ agrees, at least in the classical case, with the homology of the Morse complex abstractly associated to the flow category $\scr F$:

    \begin{PROP}\label{prop:hom}
        If $\scr F$ is a classical flow category (i.e. so that $\scr C_i = p_i$ are points), equipped with a stable normal framing, then the unstable CJS construction is naturally homeomorphic to a CW complex. Further, if $\scr F$ is equipped with a Maslov grading $\mu$, then the cellular complex $C^{\rm cell}_*$ of the stable CJS construction, defined as the negative shift of the cellular complex of the unstable CJS construction by \ref{eq:desusp}, is isomorphic to the Morse complex $C^{\rm Morse}_*$ defined by $C^{\rm Morse}_k := \bb Z\<p_i : \mu(i)=k\>$, with
        \[ d(p_i) = \sum_{\mu(j) = \mu(i)-1} (\# \scr M_{i,j}) \cdot p_j, \]
        where $\#$ indicates the signed count with respect to the orientation on $\scr M_{i,j}$ coming from the data of their stable normal framings.
    \end{PROP}

    We expect the same result to hold true more generally, for any suitable notion of Morse-Bott homology, but we do not endeavor to prove this in the present paper.

    \begin{proof}
        The filtration \ref{eq:cjsfilt} could be said to be cellular, in that its graded quotients are spheres. However, these spheres are not necessarily arranged in ascending order of dimension. To fix this, let us denote by $X_{\le k}$ the union of all the $i$-summands with $\mu(i) \le k$. To see that this filtration respects the quotient \ref{eq:eqrel}, note that $\scr M_{i,j}$ is empty unless $\mu(j) < \mu(i)$; so the identification maps on the boundary of the $i^{\rm th}$ summand always land into $j^{\rm th}$ summands with $\mu(j) < \mu(i)$. In particular, this yields a filtration of the unstable CJS construction. The graded quotient $X_{\le k} / X_{\le k-1}$ is a wedge of spheres of the same dimension:
        \[ \begin{aligned}
            \dim V_{\min I, i} + {\rm rk\;} \xi_i + \dim \bb F_{i, \max I} = \dim V_{\min I, \max I} - \mu(\max I) + {\rm rk\;} \xi_{\max I} + k.
        \end{aligned} \]
        The shift inherent in the \emph{stable} CJS construction gets rid of all the extra terms, so that now $X_{\le k}/X_{\le k-1}$ becomes a wedge of spheres of dimension $k$ after the shift. In particular, the underlying graded abelian groups of $C^{\rm cell}_*$ and $C^{\rm Morse}_*$ agree. That the differentials agree follows from the fact that the degree of the Pontrjagyn-Thom collapse map associated to an oriented 0-manifold precisely agrees with its signed count.
    \end{proof}

    \begin{THM}[Main result]\label{thm:main}
        The stable Cohen-Jones-Segal construction applied to the flow category $\scr F_{\rm MB}$ associated to a generic Morse-Bott flow $(M, f, \Phi)$ and $E$-twisted stable normal framing, the CJS construction recovers the Thom spectrum $M^E := \Sigma^{-{\rm rk\;} E} {\rm Th}(E)$.
    \end{THM}

    \begin{proof}
        Recall that the $E$-twisted framing on $\scr F_{\rm MB}$ is the restriction of a stable normal framing on $\scr F^{\rm augm}_{\rm MB}$ relative to $\xi_\infty = E$. So the Pontrjagyn-Thom maps involving the extra index $\infty$ are of the form ${\rm PT}_{i,\infty} : {\rm Th}(\xi_i) \wedge \bb E_{i,\infty}^+ \to V_{i,\infty}^+ \wedge {\rm Th}(E).$
        We use these ${\rm PT}_{i, \infty}$ to construct a map out of the unstable CJS construction \ref{eq:cjs} for $\scr F_{\rm MB}$, smashed with the sphere $\bb E_{n,\infty}^+$, into $V_{1,\infty}^+ \wedge {\rm Th}(E)$. On the $i^{\rm th}$ summand, we define the desired map to be
        \begin{equation*}\numeq\label{eq:compmap}
            \begin{gathered}
                V_{1,i}^+ \wedge {\rm Th}(\xi_i) \wedge \bb F_{i,n}^+ \wedge \bb E^+_{n,\infty} \cong V_{1,i}^+ \wedge {\rm Th}(\xi_i) \wedge \bb E_{i,\infty}^+ \\
                \overset{{\rm id} \wedge {\rm PT}_{i,\infty}}{-\!\!\!-\!\!\!-\!\!\!\longrightarrow} V_{1,i}^+ \wedge V_{i,\infty}^+ \wedge {\rm Th}(E) \cong V_{1, \infty}^+ \wedge {\rm Th}(E).
            \end{gathered}
        \end{equation*}
        That this descends to the quotient follows again by the compatibility \ref{eq:ptassoc} with the boundary inclusions. Since $\phi_{i,\infty}$ is an isomorphism, and $\dim \scr W_n = \dim \scr C_n$, it follows that $\dim \bb E_{n,\infty} = \dim V_{n,\infty} + {\rm rk\;} E - {\rm rk\;} \xi_n$. Together with the fact that $\mu(n)=0$, it follows that, after desuspending \ref{eq:compmap} by $\dim V_{1,\infty} + {\rm rk}\; E$, we get a well-defined map from the stable CJS construction to the Thom spectrum $\Sigma^{-{\rm rk}\; E}{\rm Th}(E)$.

        It remains to prove that it is a stable weak equivalence. We do so by an inductive 5-lemma argument similar to that of {\sc Prop.\,\ref{prop:hinv}}. There is again a filtration $\{X_i\}$ of the unstable CJS construction $X$, see \ref{eq:cjsfilt}, and accordingly also a filtration $\{M_i\}$ on $M$, where $M_i$ is the subset consisting of all points in $M$ below the level-set $L_{i-1,i}$ (and $M = M_1$ by convention). Let $E_i$ denote the restriction of $E$ to $M_i$. The stabilized map $X \wedge \bb E_{n,\infty}^+ \to V_{1,\infty}^+ \wedge {\rm Th}(E)$ we have just constructed respects this filtration given by the $X_i$ and ${\rm Th}(E_i)$. By the 5-lemma, it remains to check that the induced map $(X_{i}/X_{i+1}) \wedge \bb E_{n,\infty}^+ \to V_{1,\infty}^+ \wedge ({\rm Th}(E_i)/{\rm Th}(E_{i+1}))$ on the graded quotients is a stable weak-equivalence.
        
        First, we show that both sides are homotopy-equivalent to a suspension of ${\rm Th}(\xi_i)$. On the left-hand side, we have $(X_i/X_{i+1}) \wedge \bb E_{n,\infty}^+ = V_{1,i}^+ \wedge {\rm Th}(\xi_i) \wedge (\bb E_{i, \infty}^+ / \partial)$, the quotient by all the boundary strata. Since $\bb E_{i,\infty}^+/\partial$ is homeomorphic to a sphere, the left-hand side is indeed a suspension of ${\rm Th}(\xi_i)$ of rank $\dim V_{1,i} + \dim \bb E_{i,\infty}$. On the right-hand side, we find that $M_i$ deformation-retracts onto $M_{i+1} \cup D_i^-$, and hence ${\rm Th}(E_i)/{\rm Th}(E_{i+1})$ deformation retracts onto the subspace ${\rm Th}(E|_{D_i^-}) / {\rm Th}(E|_{S_i^-}) \cong {\rm Th}(E|_{\scr C_i} \oplus T_i^-)$. Now, after we smash with $V_{1,\infty}^+$, and use the framing $\phi_{i,\infty} : \xi_i \oplus N_{i,\infty} \overset\sim\to V_{i,\infty} \oplus E$, the space in question is further homotopy equivalent to
        \[ V_{1,\infty}^+ \wedge {\rm Th}(E|_{\scr C_i} \oplus T_i^-) \cong V_{1,i}^+ \wedge {\rm Th}(\xi_i \oplus N_{i,\infty}|_{\scr C_i} \oplus T_i^-) \cong V_{1,i}^+ \wedge {\rm Th}(\xi_i \oplus T\bb E_{i,\infty}), \]
        where $N_{i,\infty}$ is the normal bundle of the embedding $\scr W_i \subset \scr C_i \times \bb E_{i,\infty}$.

        Now, it is clear that both sides are abstractly homotopy-equivalent, and it remains to check that the map induced on the graded quotients is indeed a stable homotopy-equivalence. To see this, first note that the map respects the ``fibers over $\scr C_i$,'' i.e. all points in the domain either map to the basepoint in the target, or if not, the $\scr C_i$-coordinate is preserved. In each such ``fiber'', the map collapses everything far away from $\scr W_i$ to a point, and in fact everything far away from $(D_i^- \setminus M_{i+1}) \subset \scr W_i$, since we took the associated graded quotient. Finally, on a small tubular neighborhood near $(D_i^- \setminus M_{i+1}) \subset \scr C_i \times \bb E_{i, \infty}$, the map just follows the isomorphism $\phi_{i,\infty}$. Hence, in each sphere ``fiber'' of the Thom space, the map is of degree 1. It follows that the entire map induces an isomorphism on homology, following the usual proof of the Thom isomorphism using the Serre spectral sequence for pairs $E_2^{p,q} = H^p(B, \scr H^q(\bb D(E_x), \bb S(E_x))) \implies H^{p+q}(\bb D(E), \bb S(E))$, and its naturality. Hence, the map between Thom spaces is a stable homotopy equivalence by the Hurewicz theorem.
    \end{proof}

\section{Functorial continuation maps between CJS constructions}\label{sec:pss}

    In this section, we explain how the Cohen-Jones-Segal construction of {\sc \S \ref{sec:cjs}} can be made functorial, generalizing the method used in the proof of {\sc Thm.\,\ref{thm:main}}. We explain an important application of this functoriality property, namely the lifting of classical Piunikhin-Salamon-Schwarz maps to the level of CJS constructions. Our exposition does not involve any symplectic geometry, but rather only formally sets up the construction.

    \begin{DEF}
        Assume we are given two graded flow categories $\scr F_I, \scr F_J$ on finite indexing sets $I$ and $J$. We define a \keywd{flow continuation} from $\scr F_I$ to $\scr F_J$ to be any extension thereof to a graded flow category $\scr F_{IJ}$ indexed over the disjoint union $I \sqcup J$ with $i < j$ for all $i \in I, j \in J$ (cf. also ``flow bimodule'' in \cite{AB}.) We use the typical notations $\scr C_i, \scr M_{i,j}, \ldots$ from before.

        If, moreover, one is given a stable normal framing of $\scr F_{IJ}$, then we claim that there is an exact triangle
        \[ \Sigma^{-1}{\rm CJS}(\scr F_I) \overset\alpha\longrightarrow {\rm CJS}(\scr F_J) \to {\rm CJS}(\scr F_{IJ}) \to {\rm CJS}(\scr F_I) \]
        of stable CJS constructions associated to the three flow categories involved. The first map $\alpha$ in the triangle is called the \keywd{continuation map} induced by the continuation of flow categories, and is defined as follows:
        \begin{equation*}\numeq\label{eq:contmap} \begin{gathered}
            {\rm CJS}^u(\scr F_I) \wedge \left(\bigcup_{j \in J} \partial^j \bb F^+_{\max I,\max J} \right) \cong \bigcup_{j \in J} \partial^j \left(\coprod_{i \in I} V_{\min I, i}^+ \wedge {\rm Th}(\xi_i) \wedge \bb F_{i, \max J}^+ \;\Big/\! \sim\right) \\ \overset{{\rm id} \wedge {\rm PT}_{i,j} \wedge {\rm id}}{-\!\!\!-\!\!\!-\!\!\!\longrightarrow} \coprod_{j \in J} V_{\min I, j}^+ \wedge {\rm Th}(\xi_j) \wedge \bb F_{j, \max J}^+ {\;\Big /\! \sim} \;\;\cong\;\; V^+_{\min I, \min J} \wedge {\rm CJS}^u(\scr F_J), 
        \end{gathered} \end{equation*}
        where by the map denoted ${\rm id} \wedge {\rm PT}_{i,j} \wedge {\rm id}$ is understood, on each $\partial^j$-boundary and $i^{\rm th}$ summand, the identification map used in the definition of the quotient \ref{eq:cjs} for ${\rm CJS}(\scr F_{IJ})$; these identification maps agree on the pairwise overlaps precisely because the original equivalence relation defining \ref{eq:cjs} was transitive. \ref{eq:contmap} indeed induces a map $\Sigma^{-1}{\rm CJS}(\scr F_I) \overset\alpha\longrightarrow {\rm CJS}(\scr F_J)$ on stable CJS constructions, because both $\bigcup_{j \in J} \partial^j \bb F^+_{\max I, \max J}$ and $V_{\min I, \min J}^+$ are spheres of the appropriate dimensions, as can be seen from a simple calculation. The cofiber of $\alpha$ is equivalent to ${\rm CJS}(\scr F_{IJ})$, because the latter is defined precisely as attaching the contractible space ${\rm CJS}^u(\scr F_I) \wedge \bb F^+_{\max I, \max J}$ along its $(\bigcup_{j \in J} \partial^j)$-boundary, i.e. the domain of \ref{eq:contmap}, to the codomain $V_{\min I, \max J}^+ \wedge {\rm CJS}^u(\scr F_J)$ of \ref{eq:contmap}.
    \end{DEF}

    \begin{RMK}\label{rmk:conthom}
        At least in the classical case $\scr C_i = p_i$, the continuation map on CJS constructions lifts the usual Morse continuation map on the level of Morse complexes $C_*^{\rm Morse}$, see {\sc Prop.\,\ref{prop:hom}}, given by 
        \[f(p_i) = \sum_{\mu(j) = \mu(i)-1} (\#\scr M_{i,j}) \cdot p_j.\]
        This can be seen by directly inspecting the degree on the graded quotients associated to the two CW filtrations on ${\rm CJS}(\scr F_I)$ and ${\rm CJS}(\scr F_J)$ respectively.
    \end{RMK}

    \begin{EXP}\label{exp:comparisonmap}
        The Morse-Bott flow category $\scr F^{\rm augm}_{\rm MB}$ of {\sc Ex.\,\ref{exp:Faugm}} is a flow continuation between the flow category $\scr F_{\rm MB}$ of {\sc Ex.\,\ref{exp:F}}, and the flow category with one critical locus being $M$ itself. Consequently one obtains a continuation map from the Cohen-Jones-Segal construction associated to $\scr F_{\rm MB}$ with the $E$-twisted framings, to the Thom spectrum $\Sigma^{-{\rm rk\;}E} {\rm Th}(E)$. We leave it to the reader to check that this agrees with the map constructed in the proof of the {\sc Main Thm.\,\ref{thm:main}}, and hence that it is a stable homotopy-equivalence. In particular, the CJS construction applied to $\scr F^{\rm augm}_{\rm MB}$ is contractible.
    \end{EXP}

    \begin{RMK}\label{rmk:rev}
        Since the construction of $\scr F_{\rm MB}$ is symmetric with respect to reversing the flow, see {\sc Rmk.\,\ref{rmk:sym}}, one can instead augment it to a flow category $\scr F^{-\rm augm}_{\rm MB}$ by adding an extra index $-\infty$ before $1$, and setting $\scr M_{-\infty, i} := \scr W^i$ (notice the upper $i$), i.e. the $\scr W_i$ of the \emph{reverse} flow $(M,-f,-\Phi)$. This will give continuation maps in the opposite direction, provided we can frame it, which we now explain in the simplified case $\scr C_i = p_i$ with $\xi_i = 0$ of {\sc Rmk.\,\ref{rmk:other}}. Given the inherent asymmetry in the notion of stable normal framing, {\sc Def.\,\ref{def:snf}}, we claim that any stable normal framing of $\scr F^{\rm augm}_{\rm MB}(M,-f,-\Phi)$ relative $\xi_\infty = E$ induces a stable normal framing of $\scr F^{-\rm augm}_{\rm MB}(M, f, \Phi)$ relative $\xi_{-\infty} = -TM - E$ (suspended enough times to become an actual vector bundle.)
        
        Indeed, for $i \neq -\infty$, we can just use the same framings $\phi_{i,j}$ (since $\xi_i=0$, there is no asymmetry.) But at $j = -\infty$, we must embed $\scr W^j$ into $M \times \bb E_{-\infty, j}$, instead of just $\bb E_{-\infty, j}$; of course, we can simply take take the old embedding into $\bb E_{-\infty, j}$, and tack on the $M$-factor via the evaluation at the free endpoint. As a result, the normal bundle picks up an extra copy of $TM$, so we have to add $-TM$ to $\xi_{-\infty}$ to compensate. Similarly, we must add $-E$ to move the copy of $E$ from $\xi_\infty$ to $\xi_j$, thereby fixing the asymmetry.
        
        A proof very similar to that of our main result shows that the continuation map gotten in this way is a stable homotopy equivalence, which we leave to the interested reader. In particular, reversing the ordering of the critical sets, while preserving the stable normal framing, results in taking the \emph{Spanier-Whitehead dual} on the level of CJS constructions. We conjecture that this is true for an arbitrary framed flow category, not just $\scr F_{\rm MB}$.
    \end{RMK}

    Next, one may ask whether the construction of continuation maps can be extended to become functorial. In other words, if one is given graded flow categories $\scr F_I, \scr F_J, \scr F_K$ indexed over $I$, $J$ and $K$ respectively, and flow continuation maps $\scr F_{IJ}$, $\scr F_{JK}$ from $\scr F_{I}$ to $\scr F_{J}$, and from $\scr F_{J}$ to $\scr F_{K}$ respectively, together with compatible stable normal framings so as to obtain continuation maps, then is there a good description of the composite
    \begin{equation*}\numeq\label{eq:comp}
        \Sigma^{-2}{\rm CJS}(\scr F_I) \to \Sigma^{-1}{\rm CJS}(\scr F_J) \to {\rm CJS}(\scr F_K)?
    \end{equation*}
    
    We claim that this composite can itself be described as a continuation map between $\scr F_{I}$ and $\scr F_{K}$, where the Maslov grading of $\scr F_{I}$ is artificially decreased by 1 (this shift is clearly necessary because of the $-2$ in the formula above.)

    \begin{DEF}
        Given flow continuations $\scr F_{IJ}$ from $\scr F_I$ to $\scr F_J$, and $\scr F_{JK}$ from $\scr F_J$ to $\scr F_K$ with the usual notations, we define the \keywd{composite flow continuation} from $\scr F_I$ to $\scr F_K$ to be given by a suitable smoothing $\scr M_{i,k}^{\rm sm}$ of the union
        \[ \scr M_{i,k} := \bigcup_{j \in J} \scr M_{i,j} \times_{\scr C_j} \scr M_{j,k} \]
        for all $i \in I, k \in K$, where the common subspaces
        \[ \scr M_{i,j} \times_{\scr C_j} \scr M_{j,k} \supset \scr M_{i,j} \times_{\scr C_j} \scr M_{j,j'} \times_{\scr C_{j'}} \scr M_{j', k} \subset \scr M_{i,j'} \times_{\scr C_{j'}} \scr M_{j',k}\]
        get identified. For this to make sense, we need to say a few more words.
        
        First, for the fiber-products to be well-defined, the transversality of all source and target maps is presupposed; if this is not true, transversality can be achieved by small perturbation of the source and target maps to $\scr C_j$. Second, it is not immediate from the definition in what way this is a manifold. It is not hard to see that it is a topological manifold with $(I \sqcup K)$-corners, because the pieces always locally meet like the closed boundary strata $\{\partial^j \bb R_{\ge 0}^{I\sqcup J\sqcup K} \times V\}_{j \in J}$ which form the $\bigcup_{j \in J}$-boundary of $\bb R_{\ge 0}^{I\sqcup J\sqcup K} \times V$ for some vector space $V$, a topological manifold with $I \sqcup K$-corners. To endow these moduli spaces with smooth structures and framings, we can actually take inspiration from the CJS construction. Consider embeddings $\iota_{i,j} : \scr M_{i,j} \hookrightarrow \scr C_i \times \bb E_{i,j}$ and $\iota_{j,k} : \scr M_{j,k} \hookrightarrow \scr C_j \times \bb E_{j,k}$, which together give embeddings $\scr M_{i,j} \times_{\scr C_j} \scr M_{j,k} \hookrightarrow \scr C_i \times \bb E_{i,j} \times \bb E_{j,k}$, see \ref{eq:fibmap}. These agree on the triple fiber-product overlaps, so together they give an embedding
        \[ \iota_{i,k} : \scr M_{i,k} \hookrightarrow \scr C_i \times \bigcup_{j \in J} \partial^j\bb E_{i,k}, \]
        where $\bb E_{i,k} := \bb E_{i,j} \times \bb R_{\ge 0}^{\{j\}} \times \bb E_{j,k}$ for some (any) $j \in J$. In order to smoothe these spaces, let us make the further assumption that the original embeddings $\iota_{i,j}$ and $\iota_{j,k}$ are \emph{collar-like} near the closed boundary strata, meaning that for any intermediary index $i < u < j$, the embedding of $\scr M_{i,j}$ in $\scr C_i \times \bb E_{i,j}$ near the $u$-boundary looks like the embedding of $\partial^u \scr M_{i,j}$ in $\scr C_i \times \bb E_{i,u} \times \{0\} \times \bb E_{u,j}$ times a small interval $[0, \epsilon]$. It is not hard to make any given system of embeddings collar-like by a perturbation. Then, one may thicken $\iota_{i,k}$ from above to obtain a collar-like submanifold $\scr M_{i,k}^{\rm thick}$ of a neighborhood of $\scr C_i \times \bigcup_{j \in J} \partial^j\bb E_{i,k}$ inside $\scr C_i \times \bb E_{i,k}$. If $\{x_j\}_{j \in J}$ denote the $\bb R_{\ge 0}$-coordinates in the $J$-directions, then we may intersect this submanifold with the hypersurface $\widetilde{\bb E}_{i,k} := \{\prod_{j \in J} x_j = \epsilon\} \subset \bb E_{i,k}$ (for sufficiently small $\epsilon > 0$), in order to obtain a \emph{smoothing} $\scr M^{\rm sm}_{i,k}$ of $\scr M_{i,k}$, together with a smooth embedding into the new $\scr C_i \times \widetilde{\bb E}_{i,k}$. The normal framings can likewise be extended up to contractible spaces of choices over the thickening, resulting in framings of the normal bundles of the $\scr M_{i,k}^{\rm sm} \subset \scr C_i \times \widetilde{\bb E}_{i,k}$. Notice that $\dim \scr M_{i,k}^{\rm sm}$ is one less than expected, so we must decrease the Maslov grading of $\scr F_I$ by one. Finally, $\scr M_{i,k}^{\rm thick}$ gives a homotopy (via the Pontrjagyn-Thom construction) between the composite \ref{eq:comp}, and the continuation map induced by $\scr M^{\rm sm}_{i,k}$, which can be seen by direct inspection of the map \ref{eq:contmap}.
    \end{DEF}

    \begin{EXP}
        An important application of this, motivated by symplectic geometry, is the case in which the flow continuation $\scr F_{IJ}$ is the one explained in {\sc Ex.\,\ref{exp:comparisonmap}}, and the flow category $\scr F_K$ is any other flow category (which one should think of as being associated to some Floer data.) More concretely, let us use Roman indices $\scr M_{i,j}$ to denote Morse moduli spaces of $\scr F_I$, and Greek indices $\scr M_{\alpha, \beta}$ to denote ``Floer'' moduli spaces of $\scr F_K$. Now, a flow continuation between $\scr F_J$ and $\scr F_K$ consists of moduli spaces $\scr W^\alpha$, with an evaluation map to the manifold $M$ (as motivation for this, consider moduli of holomorphic punctured spheres $\bb C = \bb C_\infty \setminus \{\infty\}$ converging asymptotically to a Floer critical point near the puncture $\infty$, with evaluation back to the manifold itself via evaluation at $0$.) The standard way of comparing the Floer homology associated to $\scr M_{\alpha, \beta}$ to the homology of $M$ is to define mixed moduli spaces
        \begin{equation*}\numeq\label{eq:mixedmoduli}
            \scr M_{i,\alpha} := \scr W_i \times_M \scr W^\alpha,
        \end{equation*}
        which in the Floer-theoretic analogy parameterize Morse flow lines intersecting the punctured holomorphic spheres, and constructing a continuation map from these. It is clear this is precisely the composite flow continuation between that from $\scr F_I$ to $\scr F_J$ (which induces an equivalence of CJS constructions), and that from $\scr F_J$ to $\scr F_K$. In particular, one obtains a commutative triangle
        \[ \begin{tikzcd}[column sep=tiny]
            {\rm CJS}\left(\text{Morse } \scr M_{i,j}\right) \ar[rd, "\sim", swap]\ar[rr, "{\rm cl.\; PSS}"] && {\rm CJS}\left(\text{Floer } \scr M_{\alpha, \beta}\right) \\
            & M^E \ar[ru, "{\rm simpl.\; PSS}", swap]
        \end{tikzcd} \]
        in the stable homotopy category, where the horizontal map is induced by the mixed moduli \ref{eq:mixedmoduli}, which we call the \keywd{classical Piunikhin-Salamon-Schwarz} map. This shows that there is a simpler way to compare the ``Floer'' homotopy type ${\rm CJS}(\scr M_{\alpha, \beta})$ to the original manifold $M$, without having to construct the intermediary \ref{eq:mixedmoduli}. We call it the \keywd{simplified PSS map}, since its definition does not use the Morse theory of $M$ at all. Our main result {\sc Thm.\,\ref{thm:main}}, together with {\sc Rmk.\,\ref{rmk:conthom}}, implies that this simplified map lifts the usual classical PSS map on the level of homology. Finally, we note that one can similarly obtain PSS maps going in the other direction, by virtue of {\sc Rmk.\,\ref{rmk:rev}}.
    \end{EXP}

\section{Relaxing stable normal framings to spectrum-valued orientations}\label{sec:or}

    An absolutely necessary ingredient in performing the Cohen-Jones-Segal construction of {\sc Def.\,\ref{def:cjs}} for a flow category $\scr F$ is the data of a \emph{stable normal framing} thereof, see {\sc Def.\,\ref{def:snf}}. The mere existence of such stable normal framings is quite a strong condition; for instance, even in the classical case $\scr C_i = {\rm pt}$, it would follow in particular that $[T\scr M_{i,j}] \in \widetilde{KO}(\scr M_{i,j})$ are all zero. For the flow categories $\scr F_{\rm MB}, \scr F_{\rm MB}^{\rm augm}$ of {\sc\S\ref{sec:smooth}}, the stable normal framings were built in {\sc\S\ref{sec:stfr}} by exploiting the specific structure of the quotient $\widetilde{\scr W_i}$. But for a general flow category arising from some Floer data, there is no a priori reason to expect that they can be stably framed; however, this notion can be relaxed in certain cases, as we show in this final section.

    Instead of using the very limited Spanier-Whitehead category like we did in {\sc\S\ref{sec:cjs}}, we instead make use of \emph{orthogonal spectra}, which we now briefly recall (although perhaps other models of highly structured spectra would work equally well.) Some good references include \cite{MMSS}, \cite{MM}, \cite{Sch}. The proficient reader may skip directly to {\sc Def.\,\ref{def:or}}.

    \begin{DEF}
        We define a category $\scr O$ enriched in $(\scr{Top}_*, \wedge, S^0)$, the symmetric monoidal category of pointed topological spaces. Its objects are finite-dimensional Euclidean spaces. Given two such spaces $V$ and $W$, let ${\rm Emb}(V, W)$ denote the space of isometric embeddings $V \hookrightarrow W$, and let $W - V$ denote the vector sub-bundle of the trivial bundle $\underline{W}$ over ${\rm Emb}(V, W)$ whose fiber over an embedding $i : V \hookrightarrow W$ is given by the orthogonal complement $i(V)^\perp$ in $W$. We define the Hom-space $\scr O(V, W) := {\rm Th}(W - V)$,
        to be the Thom space of this bundle, pointed by the point at $\infty$. Given nested embeddings $U \hookrightarrow V \hookrightarrow W$, the vector spaces $(W-V) \oplus (V-U)$ and $(V - W)$ are canonically identified; so over the whole space of embeddings, this gives multiplication maps $\scr O(V, W) \wedge \scr O(U, V) \to \scr O(U, W)$,
        which are easily seen to be associative and unital, with identity given by ${\rm id}_V \in \scr O(V, V)$.
    \end{DEF}

    \begin{DEF}
        An \keywd{orthogonal spectrum} is any pointed topological functor $\scr O \to \scr{Top_*}$.
    \end{DEF}

    To unpack this definition, an orthogonal spectrum $X$ is a collection of pointed spaces $X(V)$, together with suitably compatible pointed continuous maps $\scr O(V, W) \to {\scr Top_*}(X(V), X(W))$. By currying any such map, one obtains a map $\alpha_{V, W} : \scr O(V, W) \wedge X(V) \to X(W)$. To understand $\alpha_{V,W}$, let us restrict it to the fiber over a particular embedding $i : V \hookrightarrow W$; the map takes the form $\alpha_i : (W - i(V))^+ \wedge X(V) \to X(W)$. All these maps $\alpha_i$ for various embeddings $i : V \hookrightarrow W$ can be said to vary continuously, since they assemble together into the map $\alpha_{V, W}$. Restricting $X$ to the Euclidean spaces $\cdots \subset \bb R^k \subset \bb R^{k+1} \subset \cdots$ recovers a classical suspension spectrum.

    \begin{EXP}
        The \emph{sphere spectrum} $\bb S$ is given by the Yoneda functor $\scr O(0,-) : \scr O \to \scr{Top_*}$. Concretely, it has $\bb S(V) := V^+$, with maps $\alpha_{V, W} : \scr O(V, W) \wedge V^+ \to W^+$ induced from the isomorphism $W \cong (W - V) \oplus V$.
    \end{EXP}

    \begin{EXP}
        Given an orthogonal spectrum $X$, and a pointed space $A$, one can produce another orthogonal spectrum $X \wedge A$ given by post-composing $X : \scr O \to \scr{Top_*}$ with the smashing functor $- \wedge A : \scr {Top_*} \to \scr {Top_*}$. Particularly important is $\bb S \wedge X =: \Sigma^\infty X$, the \emph{suspension spectrum} on $X$.
    \end{EXP}

    \begin{EXP}
        Given an orthogonal spectrum $X$, and a finite-dimensional Euclidean vector space $W$, there is a new spectrum ${\rm sh}^W X$ called the \emph{$W$-shift of $X$}, given by pre-composing $X : \scr O \to \scr{Top_*}$ with the functor $W \oplus - : \scr O \to \scr O$ which takes direct sum with $W$, so that on objects one has $({\rm sh}^W X)(V) = X(W \oplus V)$.
    \end{EXP}

    A most important feature of orthogonal spectra, compared to sequential spectra, is that they can be endowed with a symmetric-monoidal smash product of spectra, with unit given by $\bb S$. Its literal definition is quite long and abstruse, so instead we formulate a universal property for it:

    \begin{DEF}
        Given orthogonal spectra $X, Y, Z$, a \emph{bilinear map} $\mu : (X, Y) \to Z$ is defined to be a collection of maps $\mu_{V,W} : X(V) \wedge Y(W) \to Z(V \oplus W)$ satisfying the commutativity
        \[ \begin{tikzcd}
            \scr O(\tilde V, V) \wedge X(\tilde V) \wedge \scr O(\tilde W, W) \wedge Y(\tilde W) \dar["\alpha^X_{\tilde V, V} \wedge \alpha^Y_{\tilde W, W}"]\rar & \scr O(\tilde V \oplus \tilde W, V \oplus W) \wedge Z(\tilde V \oplus \tilde W) \dar["\alpha^Z_{\tilde V \oplus \tilde W, V \oplus W}"] \\
            X(V) \wedge Y(W) \rar["\mu_{V, W}"] & Z(V \oplus W)
        \end{tikzcd} \]
        where the top horizontal map is induced by the canonical identification of $(V - \tilde V) \oplus (W - \tilde W)$ and $(V \oplus W) - (\tilde V \oplus \tilde W)$ on the two $\scr O$-factors, and by $\mu_{\tilde V, \tilde W}$ on the other two factors. A \keywd{smash product} of $X$ and $Y$ is any \emph{universal} bilinear map $(X, Y) \to X \wedge Y$; such a smash product exists, and is unique up to unique isomorphism. (Much like how the tensor product $M \otimes_R N$ of modules is the universal recipient of a bilinear map from $M$ and $N$.) The smash product of orthogonal spectra turns out to be symmetric-monoidal, with unit $\bb S$.
    \end{DEF}

    \begin{RMK}\label{rmk:shsm}
        Shifting and smashing do not commute, although there is a natural map
        \[ ({\rm sh}^{W_1} X_1) \wedge ({\rm sh}^{W_2} X_2) \to {\rm sh}^{W_1 \oplus W_2} (X_1 \wedge X_2). \]
        This can be seen from the universal property, by constructing new bilinear maps
        \[ ({\rm sh}^{W_1} X_1)(V_1) \wedge ({\rm sh}^{W_2} X_2)(V_2) \cong X_1(W_1 \oplus V_1) \wedge X_2(W_2 \wedge V_2) \overset{\mu}\to  (X_1 \wedge X_2)((V_1 \oplus V_2) \oplus (W_1 \oplus W_2)). \]
        Despite not being an isomorphism, it is a stable homotopy equivalence.
    \end{RMK}

    \begin{DEF}
        A \keywd{ring spectrum} is an orthogonal spectrum $R$, equipped with a unit map $\eta : \bb S \to R$ and a multiplication map $\mu : R \wedge R \to R$ that are unital and associative in the usual diagrammatical sense. A \emph{left module} $M$ over $R$ is an orthogonal spectrum $M$ equipped with action maps $\alpha : R \wedge M \to M$ which are associative and unital in the usual sense.
    \end{DEF}

    \begin{EXP}\label{exp:rmod}
        In virtue of {\sc Rmk.\,\ref{rmk:shsm}}, given a ring spectrum $R$, any shift ${\rm sh}^W R$ is a left module over $R$, via the composite
        \[ ({\rm sh}^0 R) \wedge ({\rm sh}^W R) \to {\rm sh}^{W}(R \wedge R) \overset{{\rm sh}^W\mu}\longrightarrow {\rm sh}^W R. \]
        Similarly, given any other spectrum $X$, the smash product $R \wedge X$ is a left module over $R$. The same is true regarding any colimit of $R$-modules (colimits are computed pointwise, and hence will commute with all smash products involved.)
    \end{EXP}

    Now, we introduce the appropriate relaxation of the notion of stable framing for a flow category, which will be necessary to generalize the CJS construction.

    \begin{DEF}\label{def:or}
        Let $\scr F$ a flow category with notations as in {\sc Def.\,\ref{def:flow}}, together with a coherent system of embeddings $\iota_{i,j}$ as in {\sc Def.\,\ref{def:cohemb}}, and $N_{i,j}$ normal bundles as in {\sc Rmk.\,\ref{rmk:normadd}}. If $\{R_{i,j}\}_{i<j}$ are orthogonal spectra with associative maps $R_{i,u} \wedge R_{u,j} \to R_{i,j}$, we define an \keywd{orientation of $\scr F$ relative to $\xi_i$ and $R_{i,j}$} to be a collection of morphisms of orthogonal spectra
        \[ \phi_{i,j} : \Sigma^\infty{\rm Th}(\xi_i \oplus N_{i,j}) \to R_{i,j} \wedge {\rm Th}(\xi_j), \]
        which satisfy a compatibility condition analogous to \ref{eq:snf}, namely that $\phi_{i,j}|_u$ is given by
        {\small\[ \Sigma^\infty{\rm Th}(\xi_i \oplus N_{i,u} \oplus N_{u,j}) \to R_{i,u} \wedge \Sigma^\infty{\rm Th}(\xi_u \oplus N_{u,j}) \to R_{i,u} \wedge R_{u,j} \wedge \Sigma^\infty {\rm Th}(\xi_j) \to R_{i,j} \wedge \Sigma^\infty{\rm Th}(\xi_j). \]}
        Note that all these maps must exist and satisfy this compatibility condition at the point-set level, i.e. not merely in the stable homotopy category!
    \end{DEF}

    \begin{DEF}\label{def:spcjs}
        The \keywd{unnormalized Cohen-Jones-Segal construction} associated to the data of {\sc Def.\,\ref{def:or}} above is given by
        \[\numeq\label{eq:spcjs} \coprod_{i \in I} R_{\min I, i} \wedge \Sigma^\infty {\rm Th}(\xi_i) \wedge \Sigma^\infty \bb E_{i,\max I}^+ \;\Big/\! \sim, \]
        with identification maps just like in \ref{eq:eqrel}, but with $V_{i,j}^+$ replaced by $R_{i,j}$, and ${\rm PT}_{i,j}$ replaced by the collapse map $\bb E_{i,j}^+ \to {\rm Th}(N_{i,j})$, followed by the orientation $\phi_{i,j}$, and the multiplication $R_{\min I, i} \wedge R_{i,j} \to R_{\min I, j}$. This quotient has to be interpreted appropriately, since it is performed in the category of orthogonal spectra. One can either take the construction object-wise, or alternatively express it as a colimit over all the $\partial^J$-strata of the summands involved. If $\scr F$ is graded, the \keywd{normalized Cohen-Jones-Segal construction} is obtained by artificially desuspending \ref{eq:spcjs} by \ref{eq:desusp} like in {\sc Def.\,\ref{def:cjs}}. Since desuspension is tricky to define for an orthogonal spectrum, it is best to leave it in the stable homotopy category for the purposes of this paper.
    \end{DEF}

    We present two main examples of this. The second one is in a sense strictly stronger than the first, but also a lot more speculative, since it requires substantial analytical and topological details to be made rigorous. As a result, we will only provide a sketch.

    \begin{EXP}\label{exp:bord}
        (cf. \cite{Cohen}) Consider the special case that $R$ is a ring spectrum, and every single $R_{i,j}$ is equal to a shift ${\rm sh}^{V_{i,j}} R$ (with Euclidean vector spaces $V_{i,j}$ satisfying $V_{i,u} \oplus V_{u,j} = V_{i,j}$) with multiplication maps $R_{i,u} \wedge R_{u,j} \to R_{i,j}$ coming from {\sc Rmk.\,\ref{rmk:shsm}} and the multiplication on $R$. In this case, all the summands of \ref{eq:spcjs} are $R$-modules, as well as all the maps used in the definition of the quotient; hence, the resulting CJS construction is also naturally an $R$-module, per {\sc Ex.\,\ref{exp:rmod}}.
        
        Of particular interest is the case that $R$ is the \keywd{bordism spectrum} associated to a given system of structure groups, e.g. $MO, MSO, MU, MSU$, etc., see \cite{Sch} for definitions. For simplicity, assume also that $\scr F$ is classical, and all the $\xi_i$ are zero. To consider the simplest example $R = MO$, we claim that any vector bundle $E \to B$ of rank $n$ admits a contractible space of orientation maps $\Sigma^\infty{\rm Th}(E) \to {\rm sh}^nMO$. Indeed, since $({\rm sh}^{n} MO)(V) = MO(\bb R^n \oplus V)$ is the Thom space on the universal rank-$(n + \dim V)$ bundle on $BO_{n + \dim V}$, we may simply pick a classifying map $B \to BO_n$ (which can be done up to contractible space of choices), and define the map on spectra to be the Thomification of the map of bundles:
        \[ \smash{{\rm Th}(E) \wedge V^+ \to MO(\bb R^n) \wedge V^+ \to MO(\bb R^n \oplus V) = ({\rm sh}^n MO)(V).} \]
        Similarly, one can orient the flow category $\scr F$ relative to shifts of $MO$ inductively on $|j-i|$, at each stage extending the classifying map from the union of all the lower boundary strata to the whole $\scr M_{i,j}$. In particular, if no stable framings are available for $\scr F$, one can still always get an $MO$-module out of it; since $MO$ is isomorphic as an $\bb E_\infty$-ring spectrum to the Eilenberg-MacLane spectrum ${\rm H}\bb F_2[x_i | i \neq 2^k-1, {\rm deg}(x_i) = i]$, one obtains an element of the derived category of this infinite polynomial ring.

        The other examples $MSO, MU$, etc. work similarly, but they do require some extra data in order to orient $\scr F$. For instance, $MSO$ requires orientations (in the classical sense) of the vector bundles $N_{i,j}$, compatible under the identifications $N_{i,u} \oplus N_{u,j} = N_{i,j}$. This is equivalent to orienting the tangent spaces $T\scr M_{i,j}$ in a manner that respects the concatenation maps $T\scr M_{i,j}|_{\partial^u} = T\scr M_{i,u} \oplus \bb R \oplus T\scr M_{u,j}$. Likewise, $MU$ requires stable normal complex structures which are compatible with concatenation. We end by noting that $\bb S$ itself is the bordism spectrum for stably framed manifolds, and that any framing of a rank-$n$ vector bundle yields an ${\rm sh}^n \bb S$-orientation thereof. Hence, the construction of {\sc\S\ref{sec:cjs}} applied to a stable normal framing gives the same result as the spectrum-valued one of {\sc Def.\,\ref{def:spcjs}}.
    \end{EXP}

    \begin{RMK}
        At least in the setup of {\sc Ex.\,\ref{exp:bord}} above, we point out that all of our work in {\sc\S\ref{sec:pss}} of building continuation maps carries through in the homotopy category of $R$-modules. The only difference is that, in \ref{eq:contmap}, the last isomorphism should be replaced by an $R$-equivalence going the other way, due to the issue explained in {\sc Rmk.\,\ref{rmk:shsm}}; this means that for composite continuation maps, one has to compose spans, but this is not difficult. All the continuation maps are $R$-module maps.
    \end{RMK}

    \begin{EXP}\label{exp:twist} (Speculative)
        We show an example where allowing $R_{i,j}$ to be more than mere shifts of a ring spectrum $R$ could be interesting. Assume again for simplicity that the flow category $\scr F$ is classical, i.e. $\scr C_i = \{p_i\}$, and that the moduli spaces come from some particular Floer theory with no compactness or transversality issues. More specifically, suppose that $p_i$ are all points in a configuration space $X$ (e.g. $\tilde{\scr L} M$ in Hamiltonian Floer theory), and that $\scr M_{i,j}$ describes (modulo the translation $\bb R$-action) those paths $\gamma$ inside $X$ from $p_i$ to $p_j$ on which a certain non-linear 1$^{\rm st}$ order elliptic operator $D_\gamma$ vanishes. Further, let us assume that its linearization $D_\gamma^{\rm lin}$ is of the form
        \[\numeq\label{eq:dlin} D_\gamma^{\rm lin} = \frac d{ds} + A_{\gamma(s)} \]
        on $Y \times \bb R$ for some closed manifold $Y$ (e.g. $S^1$ in Hamiltonian floer theory), where $\{A_x\}_{x \in X}$ is a sufficiently nicely-varying family of \emph{self-adjoint}, at least up to lower-order terms, elliptic operators on $Y$ (e.g. $A_x$ has principal symbol $J\frac d{dt}$ in Hamiltonian Floer theory.) The fact that there are no transversality issues translates into the statement that ${\rm Cok}(D_{\gamma}^{\rm lin}) = 0$. On the other hand, the kernel recovers the tangent space, plus the $\bb R$-translation factor:
        \[ {\rm Ker}(D_\gamma^{\rm lin}) = T\scr M_{i,j} \oplus \bb R.\]
        Now, the key observation is that $D_\gamma^{\rm lin}$ can be defined over \emph{all paths} $\gamma$ from $p_i$ to $p_j$, even those which do not solve the non-linear equation $D_\gamma u = 0$. The difference is that now the cokernel may be non-zero; still, $D_{\gamma}^{\rm lin}$ will be Fredholm, provided that $A_{p_i}$ are all non-degenerate self-adjoint operators, by work of Atiyah-Patodi-Singer \cite{APS}.

        Hence, if we denote the space of all paths from $p_i$ to $p_j$ in $X$ by $\scr P_{i,j}$ (perhaps with some regularity condition imposed), there is a map $D^{\rm lin}_{i,j} : \scr P_{i,j} \to \scr{Fred}$ into the space of Fredholm operators, which is know by Atiyah-J\"anich \cite{At}, \cite{JanK} to be a classifying space for $K$-theory. Hence, there is a stable bundle ${\rm Ind}(D^{\rm lin}_{i,j}) \in KO(\scr P_{i,j})$ which, when restricted to $\scr M_{i,j} \subset \scr P_{i,j}$ recovers $T\scr M_{i,j} \oplus \bb R$. In particular, the normal bundle $N_{i,j}$ is the restriction of $-{\rm Ind}(D^{\rm lin}_{i,j})$, up to a shift of $\dim \bb E_{i,j} + 1$. So if we define
        \[ R_{i,j} := \Sigma^{\dim \bb E_{i,j} + 1} {\rm Th}(-{\rm Ind}(D^{\rm lin}_{i,j})), \]
        the bundle $N_{i,j}$ is naturally oriented in $R_{i,j}$. Further, it is sensible that there should exist multiplication maps $R_{i,u} \wedge R_{u,j} \to R_{i,j}$ given by linearized gluing for operators of the form \ref{eq:dlin}, and moreover that one can orient $\scr F$ relative to the $R_{i,j}$. The resulting CJS construction would then be a module over the ring spectrum $R_{\min I, \min I}$, by the same logic.
        
        Thanks to the Atiyah-Patodi-Singer index theorem, the various ${\rm Ind}(D_{i,j}^{\rm lin})$ are computable via some algebraic topology from the abstract symbol of the operator, so that in the end no analysis is involved in the definition of $R_{i,j}$. However, it seems unlikely that one would be able to define these $R_{i,j}$ as orthogonal spectra, with strictly associative product maps; rather, a probably more appropriate setting is that of $\infty$-categories, since the various gluing isomorphisms involve lots of contractible spaces of choices. To our knowledge, no truly $\bb A_\infty$-version of the APS index theorem has ever been proven. But speculatively, if $\{A_x\}_{x \in X}$ is considered as a map $\sigma : X \to \scr{Fred^{sa}}$, the space of self-adjoint Fredholm operators with infinitely many positive and negative eigenvalues, then given any basepoint $p \in X$ with $A_p$ non-degenerate, the ring spectrum $R_{p,p} := {\rm Th}(-{\rm Ind}(D^{\rm lin}_{p,p}))$ should be the Thom $\bb A_\infty$-ring spectrum ${\rm Th}(\Omega\sigma)$, where $\Omega\sigma$ is the stable bundle induced by
        \[ \Omega_p X \to \Omega \scr{Fred^{sa}} \overset\sim\to \scr{Fred}, \]
        where the latter equivalence is proved by Atiyah-Singer \cite{AS} as part of their proof of Bott periodicity ($\scr{Fred^{sa}}$ classifies $KO^1$.) By work of \cite{HM}, the data of a module over ${\rm Th}(\Omega\sigma)$ can more conveniently be expressed as a \emph{twisted spectrum} over $X$, i.e. a section of a certain bundle of stable $\infty$-categories over $X$, with fiber isomorphic to the $\infty$-category of spectra. Either way, this conjectural CJS construction we have proposed should recover any weaker version obtained via {\sc Ex.\,\ref{exp:bord}}, since orientations in bordism spectra are usually gotten by factoring the classifying map $\Omega\sigma$ through some stronger classifying space, e.g. $BSO, BU$, etc., thereby giving a map from ${\rm Th}(\Omega\sigma)$ to $R = MSO, MU$, etc. Then one can tensor the CJS construction over ${\rm Th(\Omega\sigma)}$ with $R$ to recover the result of {\sc Ex.\,\ref{exp:bord}}.
    \end{EXP}

    \bibliographystyle{abbrvurl} 
    \bibliography{MorseCJS}{} 

\end{document}